\documentclass[12pt,reqno]{amsart}

\usepackage{a4wide}

\pdfoutput=1

\usepackage[utf8]{inputenc}
\usepackage[english]{babel}
\usepackage{amsmath,amsfonts,amsthm,amssymb,amscd,amsbsy}
\usepackage[all]{xy}
\usepackage{graphicx}
\usepackage{euscript}
\usepackage{mathtext}
\usepackage{upgreek}
\usepackage{hyperref}
\hypersetup{
    colorlinks=true,
    linkcolor=blue,
    filecolor=magenta,      
    urlcolor=cyan,
}
\usepackage{euscript}
\usepackage{mathtools}
\usepackage{microtype}
\usepackage{setspace}
\usepackage{fancyhdr}
\usepackage{pict2e}
\usepackage{mathrsfs}
\usepackage[shortlabels]{enumitem}
\usepackage{stackengine}
\usepackage{changepage}
\usepackage{parskip}
\usepackage{tikz}
\usetikzlibrary{arrows}
\usepackage{lmodern}
\usepackage{csquotes}
\usepackage{chngcntr}
\usepackage{etoolbox}
\usepackage{expl3}
\usepackage{pgfkeys}
\usepackage{pgfopts}
\usepackage{xparse}
\usepackage{xstring}
\usepackage[
mark=o,
root-radius=.035cm,
edge-length=0.46cm,
indefinite-edge-ratio=2,
indefinite-edge={
draw=black,fill=white,thin,densely dashed}]{dynkin-diagrams}

\counterwithout{equation}{section}

\usepackage{xpatch}
\xapptocmd\normalsize{%
 \abovedisplayskip=10pt plus 1pt minus 3pt
 \abovedisplayshortskip=1pt plus 3pt
 \belowdisplayskip=10pt plus 2pt minus 3pt
 \belowdisplayshortskip=8pt plus 3pt minus 2pt
}{}{}

\numberwithin{equation}{section}

\usepackage[
backend=biber,
style=ieee,
sorting=nyt,
giveninits=true,
dashed=false,
maxbibnames=99
]{biblatex}
\DeclareSymbolFont{largesymbols}{OMX}{cmex}{m}{n}

\newtheoremstyle{theoremstyle}
  {0mm} 
  {0mm} 
  {\itshape} 
  {} 
  {\bfseries} 
  {.} 
  {.5em} 
  {} 

\theoremstyle{theoremstyle}
\newtheorem{proposition}{Proposition}
\newtheorem{lemma}[proposition]{Lemma}
\newtheorem{theorem}[proposition]{Theorem}

\newtheorem{independentcorollary}[proposition]{Corollary}
\newtheorem{corollary}{Corollary}[proposition]
\newtheorem{fact}[proposition]{Fact}

\newtheorem{corollaryinner}{Corollary}[proposition] 
\newenvironment{sepcorollary}[1][]
 {
  \if\relax\detokenize{#1}\relax
  \else
    \ifcsname #1-used\endcsname
      \expandafter\xdef\csname #1-used\endcsname{\the\numexpr\csname #1-used\endcsname+1}%
    \else
      \expandafter\gdef\csname #1-used\endcsname{1}%
    \fi
    \renewcommand{\thecorollaryinner}{\ref*{#1}.\csname #1-used\endcsname}%
  \fi
  \corollaryinner
 }
 {\endcorollaryinner}

\newtheoremstyle{examplestyle}
  {0mm} 
  {0mm} 
  {} 
  {} 
  {\bfseries} 
  {.} 
  {.5em} 
  {} 

\theoremstyle{examplestyle}
\newtheorem{remark}[proposition]{Remark}

\newtheorem{definition}[proposition]{Definition}

\newtheoremstyle{remarks}
  {0mm} 
  {0mm} 
  {} 
  {} 
  {\bfseries} 
  {.} 
  {-.5em} 
  {} 

\theoremstyle{remarks}

\renewcommand{\Im}{\hspace{0.08em}\mathrm{Im}\hspace{0.04em}}
\newcommand{\Ker}{\mathrm{Ker} \hspace{0.04em}}
\newcommand{\ccdot}{\hspace{-1.2pt} \cdot}

\newcommand{\rank}{\mathrm{rank} \hspace{0.1em}}

\newcommand{\Id}{\mathrm{Id}}

\newcommand{\Aut}{\mathrm{Aut} \hspace{0.06em}}

\newcommand{\Inn}{\mathrm{Inn}}
\newcommand{\Hol}{\mathrm{Hol}}
\newcommand{\DD}{\mathrm{DD}}
\newcommand{\Lie}{\mathrm{Lie} \hspace{0.06em}}

\newcommand{\defeq}{\vcentcolon=}

\newcommand{\Zz}{\mathbb{Z}}
\newcommand{\Rl}{\mathbb{R}}
\newcommand{\Cx}{\mathbb{C}}
\newcommand{\Hq}{\mathbb{H}}

\newcommand{\Ad}{\mathrm{Ad} \hspace{0.06em}}
\newcommand{\ad}{\mathrm{ad} \hspace{0.06em}}
\newcommand{\mk}[1]{\mathfrak{{#1}}}

\newcommand{\mr}[1]{\mathrm{{#1}}}

\newcommand{\vspan}{\mathrm{span}}

\newcommand{\wt}[1]{\widetilde{{#1}}}

\newcommand{\Der}{\mathrm{Der}}
\newcommand{\set}[1]{\hspace{-0.8pt} \left \{ \hspace{0.03em} {#1} \hspace{0.03em} \right \}}

\newcommand{\cross}[2]{\langle \hspace{0.1em} {#1} \hspace{0.1em} | \hspace{0.1em} {#2} \hspace{0.1em} \rangle \hspace{0.02em}}
\newcommand{\restr}[2]{{\left.\kern-\nulldelimiterspace #1 \vphantom{\big|} \right|_{#2}}}

\newcommand{\GL}{\mathrm{GL}}

\renewcommand{\O}{\mathrm{O}}

\newcommand{\W}{\mathrm{W}}

\newcommand\altxrightarrow[2][0pt]{\mathrel{\ensurestackMath{\stackengine%
  {\dimexpr#1-7.5pt}{\xrightarrow{\phantom{#2}}}{\scriptstyle\!#2\,}%
  {O}{c}{F}{F}{S}}}}
\newcommand{\isoto}{\altxrightarrow[1pt]{\sim}}

\makeatletter
\newcommand{\extp}{\@ifnextchar^\@extp{\@extp^{\,}}}
\def\@extp^#1{\mathop{\bigwedge\nolimits^{\!#1}}}
\makeatother

\makeatletter
\DeclareRobustCommand{\loplus}{\mathbin{\mathpalette\dog@lsemi{+}}}
\DeclareRobustCommand{\roplus}{\mathbin{\mathpalette\dog@rsemi{+}}}

\newcommand{\dog@rsemi}[2]{\dog@semi{#1}{#2}{-90,90}}
\newcommand{\dog@lsemi}[2]{\dog@semi{#1}{#2}{270,90}}
\newcommand{\dog@semi}[3]{%
  \begingroup
  \sbox\z@{$\m@th#1#2$}%
  \setlength{\unitlength}{\dimexpr\ht\z@+\dp\z@\relax}%
  \makebox[\wd\z@]{\raisebox{-\dp\z@}{%
    \begin{picture}(1,1)
    \linethickness{\variable@rule{#1}}
    \roundcap
    \put(0.5,0.5){\makebox(0,0){\raisebox{\dp\z@}{$\m@th#1#2$}}}
    \put(0.5,0.5){\arc[#3]{0.5}}
    \end{picture}%
  }}%
  \endgroup
}
\newcommand{\variable@rule}[1]{%
  \fontdimen8  
  \ifx#1\displaystyle\textfont3\else
    \ifx#1\textstyle\textfont3\else
      \ifx#1\scriptstyle\scriptfont3\else
        \scriptscriptfont3\relax
  \fi\fi\fi
}

\makeatother

\begin{document}

\title[Automorphisms of real semisimple Lie algebras]{Automorphisms of real semisimple Lie algebras and their restricted root systems}
\author{Ivan Solonenko}
\address{Department of Mathematics, King's College London, United Kingdom}
\email{ivan.solonenko@kcl.ac.uk}

\begin{abstract}
We prove that every automorphism of the restricted root system of a real semisimple Lie algebra -- when defined properly -- can be lifted to an automorphism of that Lie algebra. In particular, this can be applied to automorphisms of the Dynkin diagram of the restricted root system. We also discuss some applications of this result to the theory of symmetric spaces of noncompact type.
\end{abstract}

\maketitle

\counterwithout{equation}{section}

\section{Introduction}\label{intro}

The correspondence between complex semisimple Lie algebras and reduced root systems is a classical and pivotal result in Lie theory. Its discovery allowed W. Killing and later E. Cartan to obtain the classification of complex simple Lie algebras by reducing it to the level of root systems. It also proves indispensable when one studies automorphisms of complex semisimple Lie algebras, as those can be -- in a sense -- reduced to automorphisms of the corresponding root system and its Dynkin diagram. One of the pillars of this theory is the Isomorphism Theorem, which basically says that isomorphisms between complex semisimple Lie algebras can be defined merely on the so-called canonical generators -- provided that the Cartan matrix is preserved. This theorem can be used to lift any automorphism of the root system of $\mk{g}$ to an automorphism of $\mk{g}$ itself. As a consequence, one can show that the outer automorphism group of $\mk{g}$ is isomorphic to the automorphism group of its Dynkin diagram.

Adjacent to the theory of complex semisimple Lie algebras is that of real semisimple Lie algebras. Those were classified by E. Cartan, but the classification is more intricate than in the complex case. Two standard classification routes are by means of Vogan or Satake diagrams: both are obtained by decorating the Dynkin diagram of the complexification of $\mk{g}$ in a certain way. There is another diagram that one can associate to a real semisimple Lie algebra, namely the Dynkin diagram of its restricted root system. The restricted root system of $\mk{g}$ loses all information about the compact ideals of $\mk{g}$ and hence is not suitable for classification. However, when treated properly, it can recover the noncompact part of $\mk{g}$, that is, the sum of all of its noncompact simple ideals. The main disadvantage of this theory when compared to the complex case is that there is no sensible analog of the Isomorphism Theorem -- there are no canonical generators to begin with. Yet, by using the known classification of real semisimple Lie algebras, it is possible to prove that every automorphism of the restricted root system of $\mk{g}$ -- when defined properly -- can be lifted to an automorphism of $\mk{g}$. This is the main purpose of the current paper. In the proof, we deliberately eschew a case-by-case consideration. Instead, after reducing the problem to simple real Lie algebras, we divide those into three different blocks -- depending on the nature of $\mk{g}$ as a real form of its complexification. This result appears to be somewhat folklore, as several authors state it without giving any reference (see, e.g., \cite[p. 11]{berndttamarufoliations} or \cite[p. 111]{murakami}).

We also give a reformulation of this result for symmetric spaces of noncompact type, as those are intimately related to noncompact semisimple Lie algebras.

This paper is organized as follows. In Section \ref{section_complex} we review some aspects of the theory of root systems and their relation to complex semisimple Lie algebras. We focus primarily on root system isomorphisms and the closely related notion of Dynkin diagram isomorphisms. In Section \ref{section_real} we look at real semisimple Lie algebras and their restricted root systems. We introduce the notion of a weighted isomorphism and prove that every weighted root system isomorphism can, in a certain sense, be lifted to a Lie algebra isomorphism. In Section \ref{symmetric_spaces} we look at the results of Section \ref{section_real} through the lens of symmetric spaces of noncompact type. We also establish some links between a symmetric space of noncompact type and its isometry Lie algebra.

\section{Root system isomorphisms and complex semisimple Lie algebras}\label{section_complex}

Practically everything discussed in this section can be found in \cite[Ch. II]{knapp} or \cite[\S 1,4]{onishchik}. We begin by reviewing some aspects of the theory of root systems. Let $(V, \Updelta)$ be a root system. Here $V$ is a finite-dimensional Euclidean real vector space and $\Updelta \subseteq V$ is the root system itself. We are not assuming $\Updelta$ to be reduced or irreducible. First, we recall the notion of isomorphism of root systems.

\begin{definition}\label{def1}
Let $(V', \Updelta')$ be another root system. A linear isomorphism $f \colon V \isoto V'$ is called a (\textbf{root system}) \textbf{isomorphism between} $\boldsymbol{(V, \Updelta)}$ \textbf{and} $\boldsymbol{(V', \Updelta')}$ (or \textbf{between} $\boldsymbol{\Updelta}$ \textbf{and} $\boldsymbol{\Updelta'}$, for brevity) if the following two conditions are satisfied:
\begin{enumerate}[(i)]
    \item $f(\Updelta) = \Updelta'$;
    \item $f$ preserves the root integers, i.e. $n_{f(\upalpha)f(\upbeta)} = n_{\upalpha \upbeta}$ for all $\upalpha, \upbeta \in \Updelta$ (here $n_{\upalpha \upbeta} = \frac{2 \cross{\upalpha}{\upbeta}}{||\upbeta||^2}$);
\end{enumerate}
If $V' = V$ and $\Updelta' = \Updelta$, we call $f$ an \textbf{automorphism of} $\boldsymbol{(V, \Updelta)}$ (or \textbf{of} $\boldsymbol{\Updelta}$, for brevity). The (finite) group of all automorphisms of $\Updelta$ is denoted by $\boldsymbol{\Aut(\Updelta)} \subseteq \GL(V)$.
\end{definition}

Note that, assuming (i), condition (ii) is automatically satisfied if $f$ is conformal (i.e., a scalar multiple of an isometry). Although a root system isomorphism does not have to be conformal in general, we are going to prove that it cannot stray too far from being one. To this end, we need the following simple

\begin{lemma}\label{rootlemma}
Let $(V, \Updelta)$ be a root system. There exists a unique (up to reordering) orthogonal decomposition $V = \bigoplus_{i=1}^k V_i$ such that $\Updelta = \bigsqcup_{i=1}^k \Updelta_i$, where $\Updelta_i = \Updelta \cap V_i$, and $(V_i, \Updelta_i)$ is an irreducible root system. Two roots $\upalpha, \upbeta \in \Updelta$ lie in the same component $\Updelta_i$ if and only if there exists a chain of roots $\uplambda_0, \uplambda_1, \ldots, \uplambda_s \in \Updelta$ with $\uplambda_0 = \upalpha, \uplambda_s = \upbeta$, such that $\cross{\uplambda_{i-1}}{\uplambda_i} \ne 0$ for $1 \leqslant i \leqslant s$.
\end{lemma}

Unsurprisingly, we call each $(V_i, \Updelta_i)$ an \textbf{irreducible component of} $\boldsymbol{(V, \Updelta)}$ and the decomposition $V = \bigoplus_{i=1}^k V_i, \Updelta = \bigsqcup_{i=1}^k \Updelta_i$ the \textbf{decomposition of} $\boldsymbol{(V, \Updelta)}$ \textbf{into its irreducible components}.

\begin{proof}[Proof of the lemma]
Introduce an equivalence relation on $\Updelta$: two roots $\upalpha, \upbeta \in \Updelta$ are equivalent if and only if they can be connected by a chain of roots $\uplambda_0, \uplambda_1, \ldots, \uplambda_s \in \Updelta$ as above. This is clearly an equivalence relation, so we can write $\Updelta = \bigsqcup_{i=1}^k \Updelta_i$ for the decomposition of $\Updelta$ into the equivalence classes. Define $V_i$ to be the linear span of $\Updelta_i$. Since $\Updelta$ spans $V$, we have $V = \sum_{i=1}^k V_i$. By construction, given $i,j \in \set{1, \ldots, k}, i \ne j$, every root $\upalpha \in \Updelta_i$ is orthogonal to every root $\upbeta \in \Updelta_j$, so $V_i \perp V_j$. Therefore, we have an orthogonal decomposition $V = \bigoplus_{i=1}^k V_i$. In particular, this implies that $\Updelta_i = \Updelta \cap V_i$ for each $i \in \set{1, \ldots, k}$. Trivially, for every subspace $W \subseteq V$, $(W, \Updelta \cap W)$ is a root system, hence so is each $(V_i, \Updelta_i)$. Note that each $\Updelta_i$ is irreducible by design. Let $V = \bigoplus_{i=1}^{k'} V'_i$ be another decomposition of $V$ as in the lemma. It follows from what we have already proven that all roots in $\Updelta \cap V'_i$ are equivalent to each other for each $i \in \set{1, \ldots, k'}$. On the other hand, if $i,j \in \set{1, \ldots, k'}, i \ne j$, no root in $\Updelta \cap V'_i$ can be equivalent to any root in $\Updelta \cap V'_j$. Consequently, the decomposition $V = \bigoplus_{i=1}^{k'} V'_i$ coincides with our constructed decomposition up to reordering of the factors, which completes the proof.
\end{proof}

Now we can prove the following result, which asserts that root system isomorphisms are "almost" conformal maps.

\begin{proposition}\label{conformal}
Let $(V, \Updelta)$ and $(V', \Updelta')$ be root systems and $f \colon V \isoto V'$ an isomorphism between them. Write $V = \bigoplus_{i=1}^k V_i, \Updelta = \bigsqcup_{i=1}^k \Updelta_i$ and $V' = \bigoplus_{i=1}^{k'} V'_i, \Updelta' = \bigsqcup_{i=1}^{k'} \Updelta'_i$ for the decompositions of $(V, \Updelta)$ and $(V', \Updelta')$ into their irreducible components. Then $k = k'$ and, after reordering $V_i$'s if needed, $f(V_i) = V'_i$ and $f(\Updelta_i) = \Updelta'_i$ for each $i \in \set{1, \ldots, k}$. Moreover, for each $i$, $\restr{f}{V_i} \colon V_i \isoto V'_i$ is a conformal map, i.e. there exists $a_i > 0$ such that $a_i \restr{f}{V_i} \colon V_i \isoto V'_i$ is an isometry.
\end{proposition}

\begin{proof}
To begin with, observe that $\upalpha \perp \upbeta \Leftrightarrow n_{\upalpha\upbeta} = 0$, so $f$ must preserve root orthogonality. From this it easily follows that $f$ preserves the equivalence relation on roots described in the proof of Lemma \ref{rootlemma}, which, in turn, implies the first assertion. For the remainder of the proof, we may assume that both $(V, \Updelta)$ and $(V', \Updelta')$ are irreducible, and we need to proof that $f$ is conformal. Pick any $\upalpha_0 \in \Updelta$ and define $a = \frac{||f(\upalpha_0)||}{||\upalpha_0||} > 0$. We will prove that $a^{-1}f$ is an isometry. Since $a^{-1}f$ already preserves the root integers, it suffices to show that it preserves the length of each root. Note that $\frac{n_{\upalpha\upbeta}}{n_{\upbeta\upalpha}} = \frac{||\upalpha||^2}{||\upbeta||^2}$ whenever $\cross{\upalpha}{\upbeta} \ne 0$. Hence, $f$ preserves the length-ratio of any pair of non-orthogonal roots. Pick any $\upbeta \in \Updelta$. According to Lemma \ref{rootlemma}, there exists a chain of roots $\uplambda_0, \uplambda_1, \ldots, \uplambda_s$ with $\uplambda_0 = \upalpha_0, \uplambda_s = \upbeta$, such that $\cross{\uplambda_{i-1}}{\uplambda_i} \ne 0$ for $1 \leqslant i \leqslant s$. We can compute:
$$
\frac{||f(\upalpha_0)||}{||f(\upbeta)||} = \frac{||f(\uplambda_0)||}{||f(\uplambda_1)||} \frac{||f(\uplambda_1)||}{||f(\uplambda_2)||} \cdots \frac{||f(\uplambda_{s-1})||}{||f(\uplambda_s)||} = \frac{||\uplambda_0||}{||\uplambda_1||} \frac{||\uplambda_1||}{||\uplambda_2||} \cdots \frac{||\uplambda_{s-1}||}{||\uplambda_s||} = \frac{||\upalpha_0||}{||\upbeta||},
$$
i.e $\frac{||f(\upbeta)||}{||\upbeta||} = \frac{||f(\upalpha_0)||}{||\upalpha_0||} = a$ or, in other words, $||a^{-1}f(\upbeta)|| = ||\upbeta||$. Consequently, $a^{-1} f$ preserves the lengths of all the roots and is an isometry, which means that $f$ is conformal.
\end{proof}

\begin{corollary}
If $(V, \Updelta)$ is an irreducible root system, then $\Aut(\Updelta) \subseteq \O(V)$.
\end{corollary}

\begin{proof}
Take any automorphism $f \in \Aut(\Updelta)$. By Proposition \ref{conformal}, there exists $a > 0$ such that $af$ is an isometry. Assume that $f$ is not orthogonal, i.e. $a \ne 1$. By replacing $f$ with $f^{-1}$ if needed, we may assume $a>1$, i.e. $f$ increases the length of any nonzero vector. But $\Updelta$ is finite, hence so is the set of lengths of all roots in $\Updelta$. Since $f(\Updelta) = \Updelta$, we arrive at a contradiction.
\end{proof}

Let $(V, \Updelta)$ be any root system. Let us look more closely at the automorphism group $\Aut(\Updelta)$. Recall that we have the Weyl group $\W(\Updelta)$ generated by the reflections $s_\upalpha$ in the root hyperplanes $\upalpha^\perp$, $\upalpha \in \Updelta$. Each $s_\upalpha$ is orthogonal and preserves $\Updelta$, hence it is an automorphism of $\Updelta$. We deduce that $\W(\Updelta) \subseteq \Aut(\Updelta)$. In fact, it is a normal subgroup, which can be easily checked on its generators: $f \in \Aut(\Updelta), \upalpha \in \Updelta \Rightarrow f s_\upalpha f^{-1} = s_{f(\upalpha)}$. The short exact sequence of groups $\W(\Updelta) \hookrightarrow \Aut(\Updelta) \twoheadrightarrow \Aut(\Updelta)/\W(\Updelta)$ splits, albeit not canonically. In order to split it, one first has to make a choice of positive roots. Pick a Weyl chamber $D \subseteq V$, let $\Updelta^+ \subseteq \Updelta$ be the corresponding subset of positive roots and $\Uplambda = \set{\upalpha_1, \ldots, \upalpha_r} \subseteq \Updelta^+$ the subset of simple roots. We denote the corresponding Dynkin diagram by $\boldsymbol{\DD}$. If a simple root $\upalpha_i$ has the property that $2\upalpha_i$ is also a root, the corresponding vertex of the Dynkin diagram is represented by two concentrated circles.

\begin{definition}\label{def2}
Let $(V', \Updelta')$ be another root system with a fixed choice of simple roots $\Uplambda' \subseteq \Updelta'^+ \subseteq \Updelta'$ and the corresponding Dynkin diagram $\DD'$. A bijection $s \colon \Uplambda \isoto \Uplambda'$ is called a (\textbf{diagram}) \textbf{isomorphism between} $\boldsymbol{\DD}$ \textbf{and} $\boldsymbol{\DD'}$ if it a graph isomorphism that preserves edge directions, the number of lines an edge consists of, and the number of circles a vertex consists of. If $V' = V, \Updelta' = \Updelta,$ and $\Uplambda' = \Uplambda$, we call $s$ an \textbf{automorphism of} $\boldsymbol{\DD}$. The group of all automorphisms of $\DD$ is denoted by $\boldsymbol{\Aut(\DD)}$.
\end{definition}

The chief example of diagram isomorphisms comes from root system isomorphisms. Suppose that $f \colon V \isoto V'$ is an isomorphism between $\Updelta$ and $\Updelta'$ such that $f(\Uplambda) = \Uplambda'$. Then $s = \restr{f}{\Uplambda} \colon \Uplambda \isoto \Uplambda'$ is clearly an isomorphism between $\DD$ and $\DD'$. This also explains why the Dynkin diagram of a root system is well-defined in the first place and does not depend on the choice of a Weyl chamber: if $\Uplambda_1 \subseteq \Updelta$ is another set of simple roots, then there exists $w \in \W(\Updelta) \subseteq \Aut(\Updelta)$ mapping $\Uplambda$ onto $\Uplambda_1$, so the corresponding Dynkin diagrams $\DD$ and $\DD_1$ are isomorphic. This construction ($f \mapsto \restr{f}{\Uplambda}$) actually exhausts\footnote{This implies that a root system is fully determined by its Dynkin diagram up to isomorphism.} all Dynkin diagram isomorphisms between $\DD$ and $\DD'$. Although this is a standard fact in the theory of root systems (see, for example, \cite[Prop. 2.66]{knapp}), we will reprove it for our own purposes in Proposition \ref{semidirect1} below.

Recall that for each $r \geqslant 1$ there exists only one irreducible nonreduced root system of rank $r$ up to isomorphism (see \cite[Prop. 2.92]{knapp}). It is denoted by $(BC)_r$ and its Dynkin diagram looks like this:
\begin{center}
\begin{dynkinDiagram}[arrow style={length=5pt}, scale=3]B{}
\dynkinRootMark{O}5
\end{dynkinDiagram}
\end{center}

\begin{remark}
Some authors who work only with reduced root systems ask $s \colon \Uplambda \isoto \Uplambda'$ in Definition \ref{def2} to preserve the Cartan matrix instead. This is equivalent to our definition for reduced root systems, as the Cartan matrix and the Dynkin diagram encode the same amount of data for such systems. However, for nonreduced root systems, our definition is stronger because the Dynkin diagram carries more (in fact, all) information about the root system in this case. For instance, the Cartan matrices of $B_r$ and $(BC)_r$ are the same, whereas their Dynkin diagrams are not -- the difference is precisely the vertex represented by two concentrated circles.
\end{remark}

It is very straightforward to compute the group $\Aut(\DD)$ for all irreducible root systems by looking at their classification:
$$
\Aut(\DD) \simeq \begin{cases}
S_3, &\text{if} \hspace{0.4em} \Updelta \simeq D_4; \\
\mathbb{Z}/2\mathbb{Z}, &\text{if} \hspace{0.4em} \Updelta \simeq A_n \, (n \geqslant 2), D_n \, (n \geqslant 5), \hspace{0.1em} \text{or} \hspace{0.2em} E_6; \\
\set{e}, &\text{otherwise}.
\end{cases}
$$
Since the set of simple roots forms a basis for the underlying space of a root system, every diagram isomorphism $s \colon \Uplambda \isoto \Uplambda'$ between $\DD$ and $\DD'$ extends uniquely to a linear isomorphism $V \isoto V'$, which we denote by the same letter. In particular, we have a natural group embedding $\Aut(\DD) \subseteq \GL(V)$ (once again, this embedding only makes sense after we fix the set of simple roots). Before we relate diagram isomorphisms to root system isomorphisms, we make a few observations. First off, note that $\Aut(\Updelta)$ acts naturally on the set of Weyl chambers of $(V, \Updelta)$. Second, let $V = \bigoplus_{i=1}^k V_i, \Updelta = \bigsqcup_{i=1}^k \Updelta_i$ be the decomposition of $(V, \Updelta)$ into its irreducible components. It is easy to see that for each $i \in \set{1, \ldots, k}$, $\Updelta^+_i = \Updelta_i \cap \Updelta^+$ is a set of positive roots for $\Updelta_i$. Consequently, we have $\Uplambda = \bigsqcup_{i=1}^k \Uplambda_i$ and $D = \prod_{i=1}^k D_i$, where $\Uplambda_i = \Uplambda \cap \Updelta_i^+$ is a set of simple roots for $\Updelta_i$ and $D_i = D \cap V_i$ is the corresponding Weyl chamber. This implies that for each $i$, the Dynkin diagram $\DD_i$ of $\Updelta_i$ is a connected component of $\DD$, and we have $\DD = \bigsqcup_{i=1}^k \DD_i$. Finally, note that $\W(\Updelta) = \prod_{i=1}^k \W(\Updelta_i)$.

\begin{proposition}\label{semidirect1}
Let $(V, \Updelta)$ and $(V', \Updelta')$ be root systems with fixed choices of simple roots $\Uplambda \subseteq \Updelta$ and $\Uplambda' \subseteq \Updelta'$.
\begin{enumerate}
    \item Given any diagram isomorphism $s \colon \Uplambda \isoto \Uplambda'$ between $\DD$ and $\DD'$, its linear extension $s \colon V \isoto V'$ is an isomorphism between $\Updelta$ and $\Updelta'$. An isomorphism $V \isoto V'$ between $\Updelta$ and $\Updelta'$ comes from a diagram isomorphism $\Uplambda \isoto \Uplambda'$ precisely when it maps $\Uplambda$ onto $\Uplambda'$.
    \item $\Aut(\DD) \subseteq \Aut(\Updelta)$. In terms of the action of $\Aut(\Updelta)$ on the set of Weyl chambers, $\Aut(\DD)$ is the stabilizer of $D$.
    \item $\Aut(\DD) = \W(\Updelta) \rtimes \Aut(\DD)$.
\end{enumerate}
\end{proposition}

\begin{proof}
Let $s \colon \Uplambda \isoto \Uplambda'$ be a diagram isomorphism between $\DD$ and $\DD'$. Recall that the Weyl group of a root system is generated by the simple reflections with respect to any choice of simple roots: $\W(\Updelta)$ is generated by $\set{s_\upalpha \mid \upalpha \in \Uplambda}$ and the same is true for $\W(\Updelta')$. Since $s(\Uplambda) = \Uplambda'$, we deduce that $s \W(\Updelta) s^{-1} = \W(\Updelta')$. On the other hand, it is well known that every root in a root system is simple (or double of a simple one) for a suitable choice of a Weyl chamber (\cite[Prop. 2.62]{knapp}). Since the Weyl group acts transitively on the set of Weyl chambers, we deduce that $\Updelta = \W(\Updelta) \cdot (\Uplambda \cup (2\Uplambda \cap \Updelta))$ (the same is true for $\Updelta')$. We know that for any $\upalpha \in \Uplambda$, $2\upalpha$ is a root if and only if $2s(\upalpha)$ is one. Altogether, we have:
$$
s(\Updelta) = s(\W(\Updelta) \cdot (\Uplambda \cup (2\Uplambda \cap \Updelta))) = s \W(\Updelta) s^{-1} \cdot s(\Uplambda \cup (2\Uplambda \cap \Updelta)) = \W(\Updelta') \cdot (\Uplambda' \cup (2\Uplambda' \cap \Updelta')) = \Updelta',
$$
so $s$ satisfies condition (i) of Definition \ref{def1}. What for condition (ii), observe that $s$ provides a bijection between the connected components of $\DD$ and those of $\DD'$. Thus, for each $i \in \set{1, \ldots, k}$, there exists $j \in \set{1, \ldots, k'}$ (clearly, $k' = k$) such that $s(\DD_i) = \DD'_j$, which means that $s(\Uplambda_i) = \Uplambda'_j$ and thus $s(V_i) = V'_j$ and $s(\Updelta_i) = \Updelta'_j$. Take any $\upalpha \in \Uplambda_i$ and let $a = \frac{||s(\upalpha)||}{||\upalpha||}$. We want to show that for every other $\upbeta \in \Uplambda_i$, $\frac{||s(\upbeta)||}{||\upbeta||} = a$. Assume that $\upbeta$ is connected to $\upalpha$ by an edge. Consider the root systems $\Updelta \cap \vspan_\Rl \set{\upalpha, \upbeta}$ and $\Updelta' \cap \vspan_\Rl \set{s(\upalpha), s(\upbeta)}$. They are both of rank 2 and we have an isomorphism between their Dynkin diagrams provided by $s$. Since there are just five root systems of rank 2 up to isomorphism, it is straightforward to see that two such root systems with isomorphic Dynkin diagrams are isomorphic. What it means for us is that $\frac{||\upbeta||}{||\upalpha||} = \frac{||s(\upbeta)||}{||s(\upalpha)||}$, hence $\frac{||s(\upbeta)||}{||\upbeta||} = \frac{||s(\upalpha)||}{||\upalpha||} = a$. Since $\DD_i$ is connected, it follows by induction that $s$ increases the lengths of all simple roots in $\Uplambda_i$ by the same factor of $a$. As we already know that it preserves the Cartan integers, we deduce that it is conformal on $V_i$ ($a^{-1}s \colon V_i \isoto V'_j$ is an isometry). But this, together with condition (i) of Definition \ref{def1}, implies that it preserves the root integers between all the roots in $\Updelta_i$ (and not only between the simple ones). Since the root integers between roots lying in different components of $\Updelta$ are all zero, we see that $s$ is a root system isomorphism, which was to be proven. The second assertion in part (1) of the proposition is trivial.

Part (2) follows from part (1), as $s \in \Aut(\Updelta)$ preserves $D$ if and only if it preserves $\Uplambda$.

Part (3) hinges on the fact that $\W(\Updelta)$ acts simply transitively on the set of Weyl chambers. It is clear from (2) that $\W(\Updelta)$ and $\Aut(\DD)$ do not intersect. On the other hand, let $f \in \Aut(\Updelta)$. There exists $w \in \W(\Updelta)$ such that $w(f(D)) = D$. But then $s = wf$ fixes $D$ and thus lies in $\Aut(\DD)$. Therefore, we have a decomposition $f = w^{-1}s, w^{-1} \in \W(\Upsigma), s \in \Aut(\DD)$. This completes the proof of part (3).
\end{proof}

Finally, we discuss some aspects of the correspondence between reduced root systems and complex semisimple Lie algebras.\label{complexnotation} Let $\mk{g}$ be a (finite-dimensional) complex semisimple Lie algebra. Pick a Cartan subalgebra $\mk{h} \subset \mk{g}$. We have the corresponding set of roots $\Updelta \subset \mk{h}^*$ and the root space decomposition $\mk{g} = \mk{h} \oplus \bigoplus_{\upalpha \in \Updelta} \mk{g}_\upalpha$. The restriction of the Killing form $B$ of $\mk{g}$ to $\mk{h}$ is nondegenerate, so it induces a $\Cx$-linear isomorphism $\mk{h} \isoto \mk{h}^*$. Write $\mk{h}^*(\Rl) \subset \mk{h}^*$ for the real span of $\Updelta$ and $\mk{h}(\Rl) \subset \mk{h}$ for its preimage under $\mk{h} \isoto \mk{h}^*$. It is a standard fact that $\mk{h}(\Rl) = \set{h \in \mk{h} \mid f(h) \in \Rl \hspace{0.5em} \forall \hspace{1pt} f \in \mk{h}^*(\Rl)}$ (hence $\mk{h}^*(\Rl)$ is the real dual of $\mk{h}(\Rl)$), and we have $\mk{h} = \mk{h}(\Rl) \oplus_\Rl i\mk{h}(\Rl)$ and $\mk{h}^* = \mk{h}^*(\Rl) \oplus_\Rl i\mk{h}^*(\Rl)$. The restriction of $B$ to $\mk{h}(\Rl)$ is positive definite and we can carry it along the isomorphism $\mk{h}(\Rl) \isoto \mk{h}^*(\Rl)$ to an inner product on $\mk{h}^*(\Rl)$. This makes $(\mk{h}^*(\Rl), \Updelta)$ into a reduced root system. Note that this inner product on $\mk{h}^*(\Rl)$ is natural and does not require any additional choices, for it comes from the Killing form, which is fully determined by Lie algebra structure of $\mk{g}$.

Now we make a choice of positive roots $\Updelta^+ \subset \Updelta$ and let $\Uplambda = \set{\upalpha_1, \ldots, \upalpha_r}$ be the corresponding set of simple roots. Write $H_i \in \mk{h}(\Rl)$ for the preimage of $\upalpha_i$ under the isomorphism $\mk{h} \isoto \mk{h}^*$ and let $h_i = \frac{2}{||\upalpha_i||^2} H_i$. Finally, make a choice of canonical generators $e_i \in \mk{g}_{\upalpha_i}, f_i \in \mk{g}_{-\upalpha_i}$. It follows from the definition of $h_i$'s that
$$
[h_i, e_j] = n_{\upalpha_j \upalpha_i} e_j, \quad [h_i, f_j] = - n_{\upalpha_j \upalpha_i} f_j.
$$
The Isomorphism Theorem asserts that if $\mk{g}'$ is another complex semisimple Lie algebra with a fixed choice of $\mk{h}',\, \Uplambda' = \set{\upalpha'_1, \ldots, \upalpha'_r} \subset \Updelta',$ and $e'_i \in \mk{g}_{\upalpha'_i}, \, f'_i \in \mk{g}_{-\upalpha'_i}, \, 1 \leqslant i \leqslant r$, such that the Cartan matrices $A = (n_{\upalpha_i \upalpha_j})_{i,j=1}^r$ and $A' = (n_{\upalpha'_i \upalpha'_j})_{i,j=1}^r$ coincide, then there exists a unique Lie algebra isomorphism $\mk{g} \isoto \mk{g}'$ sending $h_i$ to $h'_i$, $e_i$ to $e'_i$, and $f_i$ to $f'_i$ for $1 \leqslant i \leqslant r$.

Let $F \colon \mk{g} \isoto \mk{g}'$ be a Lie algebra isomorphism mapping $\mk{h}$ onto $\mk{h}'$ and write $f = (\restr{F}{\mk{h}}^*)^{-1} \colon \mk{h}^* \isoto \mk{h}'^*$. It is a matter of simple computation that $f(\Updelta) = \Updelta'$ and for any $\upalpha \in \Updelta$, $F(\mk{g}_\upalpha) = \mk{g}'_{f(\upalpha)}$. In particular, we have $f(\mk{h}^*(\Rl)) = \mk{h}'^*(\Rl)$. We will slightly abuse the notation and use the same letter $f$ for the restriction $\restr{f}{\mk{h}^*(\Rl)} \colon \mk{h}^*(\Rl) \isoto \mk{h}'^*(\Rl)$. As $F$ respects the Killing forms of $\mk{g}$ and $\mk{g}'$, it follows that $f$ is an isometry and thus a root system isomorphism\footnote{This construction also shows that different choices of a Cartan subalgebra of $\mk{g}$ lead to isomorphic root systems, as any two Cartan subalgebras differ by an inner automorphism of $\mk{g}$. The map sending $\mk{g}$ to $(\mk{h}^*(\Rl), \Updelta)$ is a 1-to-1 correspondence between the isomorphism classes of complex semisimple Lie algebras on the one hand and of reduced root systems on the other.} between $\Updelta$ and $\Updelta'$. The Isomorphism Theorem ensures that every isomorphism $\mk{h}^*(\Rl) \isoto \mk{h}'^*(\Rl)$ between $\Updelta$ and $\Updelta'$ arises in this way (this fact is essentially equivalent to the Isomorphism Theorem and is proven directly in \cite[Th. 2.108]{knapp}). As a consequence, every such isomorphism is an isometry.

If we let $\mk{g}' = \mk{g}$ and $\mk{h}' = \mk{h}$, we get a surjective Lie group homomorphism $\boldsymbol{\Uppsi} \colon N_{\Aut(\mk{g})}(\mk{h}) \twoheadrightarrow \Aut(\Updelta), \, F \mapsto f$, where $N_{\Aut(\mk{g})}(\mk{h})$ is the normalizer of $\mk{h}$ in $\Aut(\mk{g})$.

For each $\upalpha \in \Updelta$, an automorphism $\upeta \in N_{\Aut(\mk{g})}(\mk{h})$ such that $\Uppsi(\upeta) = s_\upalpha$ can be constructed explicitly (here $s_\upalpha$ is the reflection of $\mk{h}^*(\Rl)$ in the hyperplane $\upalpha^\perp$). Namely, if $e_\upalpha \in \mk{g}_\upalpha$ and $f_\upalpha \in \mk{g}_{-\upalpha}$ are such that $[e_\upalpha, f_\upalpha] = h_\upalpha$ ($\Leftrightarrow B(e_\upalpha, f_\upalpha) = \frac{2}{||\upalpha||^2}$), then $\upeta = \exp(\ad \frac{\uppi}{2}(e_\upalpha - f_\upalpha)) \in N_{\Aut(\mk{g})}(\mk{h})$ and $\Uppsi(\upeta) = s_\upalpha$ (see \cite[p. 210]{goto_grosshans}). Observe also that the Isomorphism Theorem allows to construct a section $\Aut(\DD) \hookrightarrow N_{\Aut(\mk{g})}(\mk{h})$ of $\Uppsi$ over $\Aut(\DD)$: take $s \in \Aut(\DD)$ and send it to $\hat{s} \in N_{\Aut(\mk{g})}(\mk{h})$ given by $\hat{s}(h_i) = h_{s(i)}, \hat{s}(e_i) = e_{s(i)}, \hat{s}(f_i) = f_{s(i)}$ (here we think of $s$ as a permutation of $\set{1, \ldots, r} \simeq \set{\upalpha_1, \ldots, \upalpha_r}$).

Eventually, we make the following observation. Recall that a root system isomorphism is not in general an isometry (or even a conformal map). However, as we have seen, if $\mk{g}$ and $\mk{g}'$ are complex semisimple Lie algebras with Cartan subalgebras $\mk{h}$ and $\mk{h}'$, respectively, then every root system isomorphism $\mk{h}^*(\Rl) \isoto \mk{h}'^*(\Rl)$ between $\Updelta$ and $\Updelta'$ is an isometry. The reason for this is that the inner products on $\mk{h}^*(\Rl)$ and $\mk{h}'^*(\Rl)$ are "nice" -- they come from the Killing forms of $\mk{g}$ and $\mk{g}'$. In fact, we can make the following definition. Let $(V, \Updelta)$ be any reduced root system. Take a complex semisimple Lie algebra $\mk{g}$ and a Cartan subalgebra $\mk{h} \subset \mk{g}$ such that the root system $(\mk{h}^*(\Rl), \Updelta_\mk{g})$ is isomorphic to $(V, \Updelta)$. Pick any isomorphism $\upvarphi \colon V \isoto \mk{h}^*(\Rl)$ between $\Updelta$ and $\Updelta_\mk{g}$ and carry the inner product from $\mk{h}^*(\Rl)$ to $V$ along $\upvarphi$ (as we know from Proposition \ref{conformal}, it simply amounts to renormalizing the existing inner product on $V$ by some conformal factors on the irreducible components of $(V, \Updelta)$). Suppose we have another pair $(\mk{g}', \mk{h}')$ and an isomorphism $\upvarphi' \colon V \isoto \mk{h}'^*(\Rl)$ between $\Updelta$ and $\Updelta_{\mk{g}'}$. We claim that the inner product on $V$ pulled back from $\mk{h}'^*(\Rl)$ along $\upvarphi'$ is the same. Indeed, if we write $f = \upvarphi' \circ \upvarphi^{-1} \colon \mk{h}^*(\Rl) \isoto \mk{h}'^*(\Rl)$, then $f$ is an isomorphism between $\Updelta_\mk{g}$ and $\Updelta_{\mk{g}'}$. As we discussed above, each such isomorphism is an isometry, which proves the claim. We call the inner product on $V$ constructed above \textbf{Killing}. 

\begin{independentcorollary}
Let $(V, \Updelta)$ and $(V', \Updelta')$ be reduced root systems with Killing inner products. Then every isomorphism $V \isoto V'$ between $\Updelta$ and $\Updelta'$ is an isometry. In particular, $\Aut(\Updelta) \subseteq \O(V)$.
\end{independentcorollary}

\section{Real semisimple Lie algebras and their restricted root systems}\label{section_real}

In this section we discuss restricted root systems of real semisimple Lie algebras and prove the main results. See \cite[Ch. VI]{knapp} and \cite[\S 2,3]{onishchik} for a detailed exposition of the theory of real semisimple Lie algebras.

Let $\mk{g}$ be a (finite-dimensional) real semisimple Lie algebra. Let us fix a Cartan involution $\uptheta$ of $\mk{g}$ and write $\mk{g} = \mk{k} \oplus \mk{p}$ for the corresponding Cartan decomposition. We also pick a maximal abelian subspace $\mk{a}$ of $\mk{p}$. Having this fixed, we have the set of restricted roots $\Upsigma \subseteq \mk{a}^*$ and the restricted root space decomposition $\mk{g} = \mk{g}_0 \oplus \bigoplus_{\upalpha \in \Upsigma} \mk{g}_\upalpha$, where $\mk{g}_0 = Z_\mk{g}(\mk{a}) = \mk{k}_0 \oplus \mk{a}$ and $\mk{k}_0 = Z_\mk{k}(\mk{a}) = N_\mk{k}(\mk{a})$. Note also that $\uptheta(\mk{g}_\upalpha) = \mk{g}_{-\upalpha}$ for any $\upalpha \in \Upsigma$. If we write $B$ for the Killing form of $\mk{g}$, then $B_\uptheta(X,Y) = -B(X,\uptheta Y)$ is an inner product on $\mk{g}$ (which we choose as out default inner product on $\mk{g}$ unless otherwise specified). The summands of the restricted root space decomposition are all pairwise orthogonal with respect to $B_\uptheta$, and $\mk{k}$ and $\mk{p}$ are mutually orthogonal with respect to both $B$ and $B_\uptheta$. Note that $B_\uptheta$ coincides with $B$ on $\mk{p}$ and equals $-B$ on $\mk{k}$. In a standard fashion, $\restr{B_\uptheta}{\mk{a} \times \mk{a}} = \restr{B}{\mk{a} \times \mk{a}}$ induces an isomorphism $\mk{a} \isoto \mk{a}^*$, and we carry it along this isomorphism to an inner product on $\mk{a}^*$. This way, $(\mk{a}^*, \Upsigma)$ becomes a root system, called the restricted root system of $\mk{g}$. There are three main differences with the complex semisimple case (apart from the obvious difference in how the root system is constructed):

\begin{enumerate}
    \item $\Upsigma$ does not have to be reduced.
    \item $\Upsigma$ loses all information about the compact ideals of $\mk{g}$.
    \item The dimensions of the restricted root subspaces $\mk{g}_\upalpha$ do not have to be equal to 1.
\end{enumerate}

Let us address these points individually. To begin with, (1) is not really an issue, since we know the classification of all -- not necessarily reduced -- root systems up to isomorphism. As we mentioned in Section \ref{section_complex}, the root systems $(BC)_r, \, r \geqslant 1$, exhaust the list of all irreducible nonreduced root systems up to isomorphism.

Regarding (2), we make the following observation. First of all, it is a standard fact about Cartan involutions that $\mk{k}$ is a maximal compact subalgebra of $\mk{g}$. This immediately implies that
\begin{equation*}
\mk{g} \hspace{4pt} \text{is compact}\footnote{In this case, $\uptheta = \Id_\mk{g}$ is the only Cartan involution on $\mk{g}$.} \, \Leftrightarrow \; \mk{p} = \set{0} \; \Leftrightarrow \; \mk{a} = \set{0} \; \Leftrightarrow \; \Upsigma = \varnothing.
\end{equation*}
Now, let $\mk{g} = \bigoplus_{i=1}^k \mk{g}_i$ be the decomposition of $\mk{g}$ into the sum of its simple ideals. Necessarily, $\uptheta = \uptheta_1 + \cdots + \uptheta_k$, where each $\uptheta_i$ is a Cartan involution on $\mk{g}_i$. We then have $\mk{k} = \bigoplus_{i=1}^k \mk{k}_i,\, \mk{p} = \bigoplus_{i=1}^k \mk{p}_i$, where $\mk{k}_i = \mk{k} \cap \mk{g}_i,\, \mk{p}_i = \mk{p} \cap \mk{g}_i$. Furthermore, we necessarily have $\mk{a} = \bigoplus_{i=1}^k \mk{a}_i,\, \mk{a}_i = \mk{a} \cap \mk{p}_i$, and hence $\mk{a}^* = \bigoplus_{i=1}^k \mk{a}_i^*,\, \Upsigma = \bigsqcup_{i=1}^k \Upsigma_i$, where $\Upsigma_i = \Upsigma \cap \mk{a}_i^*$ is the restricted root system of $\mk{g}_i$. Observe also that the decomposition $\mk{g} = \bigoplus_{i=1}^k \mk{g}_i$ is orthogonal with respect to the inner product $B_\uptheta$, and $B_\uptheta = B_{\uptheta_1}^1 + \cdots + B_{\uptheta_k}^k$, where $B^i$ is the Killing form of $\mk{g}_i$. Consequently, the decomposition $\mk{a}^* = \bigoplus_{i=1}^k \mk{a}_i^*$ is orthogonal as well. It is fairly easy to see from the properties of the restricted root space decomposition that a simple Lie algebra must have an irreducible restricted root system. We deduce that $\Upsigma = \Upsigma_1 \sqcup \ldots \sqcup \Upsigma_k$ is the decomposition of $\Upsigma$ into its irreducible components (except for the fact that some components $\Upsigma_i$ might be empty -- precisely when the correspondent simple ideal $\mk{g}_i$ is compact). We see that it is impossible to recover information about the compact ideals of $\mk{g}$ from $\Upsigma$. We also arrive at the following

\begin{independentcorollary}
Let $\mk{g} = \bigoplus_{i=1}^k \mk{g}_i$ as above. Then $\Upsigma$ is irreducible if and only if at most one simple ideal $\mk{g}_i$ is noncompact. In particular, if $\mk{g}$ has no nonzero compact ideals and $\Upsigma$ is irreducible, then $\mk{g}$ is simple.
\end{independentcorollary}

Let us now address (3) -- arguably, the most important difference with the complex semisimple case. We make use of the fact that $\dim_\Rl(\mk{g}_\upalpha)$ might not be equal to $1$ and call it the \textbf{multiplicity} of the root $\upalpha$ (and denote it by $\boldsymbol{\mathrm{mult}(\upalpha)}$). The idea is that we incorporate this information into the root system $\Upsigma$ itself. To this end, we make the following definition: a \textbf{weighted root system} is a root system in which every root is assigned a positive integer, called the multiplicity of the root\footnote{In this paper, we are using the letter $\Updelta$ for regular root systems and reserve $\Upsigma$ for weighted ones.}. The assignment of multiplicities to the roots in $\Upsigma$ is far from random, but since real semisimple Lie algebras seem to be the only context where root multiplicities arise, we do not make any additional assumptions in the definition. Since $\uptheta(\mk{g}_\upalpha) = \mk{g}_{-\upalpha}$, we know that $\mathrm{mult}(\upalpha) = \mathrm{mult}(-\upalpha)$ for any $\upalpha \in \Upsigma$. Below we will see that much more is true (see Theorem \ref{maintheorem}(2) and Proposition \ref{decompositions}(5)).

\begin{definition}\label{def3}
Let $\mk{g}'$ be another real semisimple Lie algebra with a Cartan involution $\uptheta'$ and a maximal abelian subspace $\mk{a}' \subseteq \mk{p}'$ fixed, and let $\Upsigma' \subseteq \mk{a}'^*$ be the corresponding restricted root system. We call a root system isomorphism $f \colon \mk{a}^* \isoto \mk{a}'^*$ between $\Upsigma$ and $\Upsigma'$ \textbf{weighted} if it preserves the root multiplicities: $\mathrm{mult}(f(\upalpha)) = \mathrm{mult}(\upalpha)$ for every $\upalpha \in \Upsigma$. We say that $(\mk{a}^*, \Upsigma)$ and $(\mk{a}'^*, \Upsigma')$ are \textbf{weighted-isomorphic} if there exists a weighted isomorphism between them. Finally, we call a weighted isomorphism $\mk{a}^* \isoto \mk{a}^*$ from $\Upsigma$ to itself a \textbf{weighted automorphism of} $\boldsymbol{(\mk{a}^*, \Upsigma)}$ (or \textbf{of} $\boldsymbol{\Upsigma}$, for short). The group of all weighted automorphisms of $(\mk{a}^*, \Upsigma)$ will be denoted by $\boldsymbol{\Aut^\mr{w}(\Upsigma)} \subseteq \Aut(\Upsigma)$.
\end{definition}

We want to relate weighted root system isomorphisms to Lie algebra isomorphisms. Let $\mk{g}$ and $\mk{g}'$ be as above and suppose that $F \colon \mk{g} \isoto \mk{g}'$ is an isomorphism such that $F \circ \uptheta = \uptheta' \circ F$ (hence $F(\mk{k}) = \mk{k}', F(\mk{p}) = \mk{p}'$) and $F(\mk{a}) = \mk{a}'$. Consider $\restr{F}{\mk{a}} \colon \mk{a} \isoto \mk{a}'$ and define $f = (\restr{F}{\mk{a}}^*)^{-1} \colon \mk{a}^* \isoto \mk{a}'^*$. Similarly to what we did in the complex semisimple case, it is easy to check that $f(\Upsigma) = \Upsigma'$ and $F(\mk{g}_\upalpha) = \mk{g}'_{f(\upalpha)}$ for every $\upalpha \in \Upsigma$. It is also clear that $F$ is an isometry with respect to the inner products $B_\uptheta$ and $B'_{\uptheta'}$, so $f$ is an isometry as well. All this implies that $f$ is a weighted isomorphism between $\Upsigma$ and $\Upsigma'$.

We can apply this construction to the situation when $\mk{g}' = \mk{g}, \uptheta' = \uptheta,$ and $\mk{a}' = \mk{a}$. Consider the Lie group $\Aut(\mk{g})$. We have a distinguished element of this group fixed, namely the Cartan involution $\uptheta \in \Aut(\mk{g})$. Consider the closed subgroup $K = Z_{\Aut(\mk{g})}(\uptheta)$ of automorphisms that commute with $\uptheta$. It can also be described as the fixed point subgroup $\Aut(\mk{g})^\uptheta$ of the involutive automorphism $C_\uptheta \colon \upeta \mapsto \uptheta \upeta \uptheta^{-1} = \uptheta \upeta \uptheta$ of $\Aut(\mk{g})$ (strictly speaking, we should write $\Aut(\mk{g})^{C_\uptheta}$, not $\Aut(\mk{g})^\uptheta$, but this is a common abuse of notation). Yet another way to describe it is  $K = N_{\Aut(\mk{g})}(\mk{k})$, the normalizer of $\mk{k}$ in $\Aut(\mk{g})$. Indeed, an automorphism $\upeta$ commutes with $\uptheta$ if and only if it respects the Cartan decomposition $\mk{g} = \mk{k} \oplus \mk{p}$. Since every automorphism of $\mk{g}$ is orthogonal with respect to the Killing form $B$ and $\mk{p}$ is the orthogonal complement of $\mk{k}$ with respect to $B$, $\upeta(\mk{k}) = \mk{k}$ automatically implies $\upeta(\mk{p}) = \mk{p}$. As we will see later, $K$ is a maximal compact subgroup of $\Aut(\mk{g})$ (see Corollary \ref{maximal_compact}). Since $\mk{g}$ is semisimple, we can identify it with the Lie algebra of $\Aut(\mk{g})$ by means of the adjoint representation: $\ad \colon \mk{g} \isoto \Der(\mk{g}) = \Lie(\Aut(\mk{g}))$. Under this identification, we have
\begin{align*}
    \Lie(K) &= \Der(\mk{g})^{\Ad(\uptheta)} \\
    &\cong \set{X \in \mk{g} \mid \ad(X) \circ \uptheta = \uptheta \circ \ad(X)} \\
    &= \set{X \in \mk{g} \mid \ad(X) = \ad(\uptheta(X))} \\
    &= \set{X \in \mk{g} \mid \uptheta(X) = X} = \mk{k}.
\end{align*}
One can show that $K = \Aut(\mk{g}) \cap \O_{B_\uptheta}(\mk{g})$ (see \cite[Lem. 2.2]{gundogan}). We are especially interested in the subgroup $N_K(\mk{a})$ of $K$. According to the previous paragraph, we can define a map $\Upomega \colon N_K(\mk{a}) \to \Aut^\mr{w}(\Upsigma), \upvarphi \mapsto (\restr{\upvarphi}{\mk{a}}^*)^{-1}$. This is easily seen to be a Lie group homomorphism.

Using some standard facts from the theory of real semisimple Lie algebras, we can prove the following:

\begin{proposition}\label{meow}
If $\mk{g}$ and $\mk{g}'$ are isomorphic real semisimple Lie algebras, then their restricted root systems are weighted-isomorphic for any choices of $\uptheta, \uptheta', \mk{a},$ and $\mk{a}'$. In particular, the restricted root system of $\mk{g}$ does not depend on the choice of $\uptheta$ and $\mk{a}$ (up to a weighted isomorphism).
\end{proposition}

\begin{proof}
Let $F \colon \mk{g} \isoto \mk{g}'$ be any isomorphism. Since any two Cartan involutions on $\mk{g}$ are conjugate by an inner automorphism, we may assume without loss of generality that $F \circ \uptheta = \uptheta' \circ F$. Moreover, any two maximal abelian subspaces of $\mk{p}$ differ by an element of $K$ (even of $K^0$; see \cite[Th. 6.51]{knapp}). Since, by definition, any element of $K$ commutes with $\uptheta$, we may assume, in addition, that $F(\mk{a}) = \mk{a}'$. But now the above construction implies that $F$ induces a weighted isomorphism between $\Upsigma$ and $\Upsigma'$.
\end{proof}

In a similar vein to the complex semisimple case, one can show that $\W(\Upsigma) \subseteq \Im(\Upomega)$. In fact, if $\upalpha \in \Upsigma$ and $X \in \mk{g}_\upalpha$ is a vector of length $\frac{\sqrt{2}}{||\upalpha||}$, then $\exp(\ad \frac{\uppi}{2}(X + \uptheta X)) \in N_K(\mk{a})$ and $\Upomega(\exp(\ad \frac{\uppi}{2}(X + \uptheta X)))$ is the reflection $s_\upalpha$ of $\mk{a}^*$ in the hyperplane $\upalpha^\perp$ (see \cite[Prop. 6.52]{knapp} for a proof). As a consequence, we see that $\W(\Upsigma) \subseteq \Aut^\mr{w}(\Upsigma)$. This means that restricted roots lying in the same orbit of $\W(\Upsigma)$ have the same multiplicities.

We make a choice of positive roots $\Upsigma^+ \subseteq \Upsigma$ and write $\Uplambda \subseteq \Upsigma^+$ for the subset of simple roots and $D \subseteq \mk{a}^*$ for the positive Weyl chamber. The sum $\mk{n} = \bigoplus_{\upalpha \in \Upsigma^+} \mk{g}_\upalpha$ is a nilpotent subalgebra of $\mk{g}$, and we have the Iwasawa decomposition $\mk{g} = \mk{k} \oplus \mk{a} \oplus \mk{n}$. Write $\DD$ for the Dynkin diagram of $\Upsigma$. Just like we did with $\Upsigma$, we say that $\DD$ is a \textbf{weighted Dynkin diagram} when we want to stress that each vertex has a positive integer assigned to it, namely its multiplicity. If a vertex is doubled, i.e. corresponds to a simple root $\upalpha$ such that $2\upalpha$ is also a root, we assign to it not just one number, but the ordered pair $(\mathrm{mult}(\upalpha), \mathrm{mult}(2\upalpha))$.

\begin{definition}\label{def4}
Let $\mk{g}'$ be another real semisimple Lie algebra with a Cartan involution $\uptheta'$, a maximal abelian subspace $\mk{a}' \subseteq \mk{p}'$, and a choice of positive roots $\Upsigma'^+ \subseteq \Upsigma'$ fixed. We call a diagram isomorphism $s \colon \Uplambda \isoto \Uplambda'$ between $\DD$ and $\DD'$ \textbf{weighted} if it preserves the vertex weights: $\mathrm{mult}(s(\upalpha)) = \mathrm{mult}(\upalpha)$ (and $\mathrm{mult}(2s(\upalpha)) = \mathrm{mult}(2\upalpha)$ in case $2\upalpha$ is a root) for every $\upalpha \in \Uplambda$. We say that $\DD$ and $\DD'$ are \textbf{weighted-isomorphic} if there exists a weighted isomorphism between them. Finally, we call a weighted isomorphism $\Uplambda \isoto \Uplambda$ from $\DD$ to itself a \textbf{weighted automorphism of} $\boldsymbol{\DD}$. The group of all weighted automorphisms of $\DD$ will be denoted by $\boldsymbol{\Aut^\mr{w}(\DD)} \subseteq \Aut(\DD)$.
\end{definition}

Since $\W(\Upsigma) \subseteq \Aut^\mr{w}(\Upsigma)$ and $\W(\Upsigma)$ acts transitively on the set of Weyl chambers in $\mk{a}^*$, we immediately get the following:

\begin{proposition}\label{beow}
Let $\mk{g}$ and $\mk{g}'$ be real semisimple Lie algebras with restricted root systems $\Upsigma$ and $\Upsigma'$, respectively. If $\Upsigma$ and $\Upsigma'$ are weighted-isomorphic (in particular, if $\mk{g}$ and $\mk{g}'$ are isomorphic), then their Dynkin diagrams are weighted-isomorphic as well for any choices of $\Upsigma^+$ and $\Upsigma'^+$. In particular, the Dynkin diagram of $\Upsigma$ does not depend on the choice of $\Upsigma^+$ (up to a weighted isomorphism).
\end{proposition}

Using the results of Section \ref{section_complex}, we can prove the converse to Proposition \ref{beow}.

\begin{proposition}
Let $\mk{g}$ and $\mk{g}'$ be real semisimple Lie algebras with restricted root systems $\Upsigma$ and $\Upsigma'$ and Dynkin diagrams $\DD$ and $\DD'$, respectively. If $\DD$ and $\DD'$ are weighted-isomorphic, then so are $\Upsigma$ and $\Upsigma'$. More specifically, if $s \colon \Uplambda \isoto \Uplambda'$ is a weighted isomorphism between $\DD$ and $\DD'$, then its unique linear extension $s \colon \mk{a} \isoto \mk{a}^*$ is a weighted isomorphism between $\Upsigma$ and $\Upsigma'$. In particular, $\Aut^\mr{w}(\DD) \subseteq \Aut^\mr{w}(\Upsigma)$.
\end{proposition}

\begin{proof}
We already know from Proposition \ref{semidirect1}(1) that $s \colon \mk{a} \isoto \mk{a}^*$ is an isomorphism between $\Upsigma$ and $\Upsigma'$, so we only need to prove that it preserves the root multiplicities. We also know from the proof of Proposition \ref{semidirect1} that $s \W(\Upsigma) s^{-1} = \W(\Upsigma')$ and $\W(\Upsigma) \cdot (\Uplambda \cup (2\Uplambda \cap \Upsigma)) = \Upsigma$. Let $\upalpha \in \Upsigma$ be any root. Take $w \in \W(\Upsigma)$ such that $w(\upalpha) \in \Uplambda \cup (2\Uplambda \cap \Upsigma)$, and write $w' = sws^{-1} \in \W(\Upsigma')$. We have:
$$
s(\upalpha) = sw^{-1}(w(\upalpha)) = w'^{-1}s(w(\upalpha)).
$$
Since $\mathrm{mult}(s(w(\upalpha))) = \mathrm{mult}(w(\upalpha))$ and elements of the Weyl group preserve root multiplicities, we get $\mathrm{mult}(s(\upalpha)) = \mathrm{mult}(\upalpha)$, so $s$ is a weighted root system isomorphism.
\end{proof}

Since both $\W(\Upsigma)$ and $\Aut^\mr{w}(\DD)$ are contained in $\Aut^\mr{w}(\Upsigma)$ and $\W(\Upsigma)$ acts transitively on the set of Weyl chambers, we immediately get the following weighted analog of Proposition \ref{semidirect1}(3):

\begin{corollary}
$\Aut^\mr{w}(\Upsigma) = \W(\Upsigma) \rtimes \Aut^\mr{w}(\DD)$.
\end{corollary}

\begin{remark}\label{woof}
To recapitulate, we know that a weighted-isomorphism class of restricted root systems yields a weighted-isomorphism class of Dynkin diagrams, and it is determined by that class. Similarly, an isomorphism class of real semisimple Lie algebras yields a weighted-isomorphism class of restricted root systems. We know that it cannot be determined by that class though, since adding a compact semisimple summand to the Lie algebra would not change the restricted root system. But it turns out that this is the only obstacle: if $\mk{g}$ has no nonzero compact ideals, it is determined up to isomorphism by its (weighted) restricted root system -- and thus by its (weighted) Dynkin diagram. The standard proof of this fact, however, is rather roundabout. One usually first classifies real semisimple Lie algebras -- compact or not -- by some other means like Satake or Vogan diagrams, and then computes explicitly the restricted root system of every Lie algebra in the classification list. It turns out that non-isomorphic real semisimple Lie algebras (without nonzero compact ideals) have non-weighted-isomorphic restricted root systems. The list of the (weighted) restricted root systems of all noncompact simple Lie algebras can be found in \cite[pp. 336-340]{submanifoldsholonomy}. Note that it is given there in the slightly different -- but equivalent -- context of irreducible symmetric spaces of noncompact type (see more on that in Section \ref{symmetric_spaces}). For example, the restricted root system of $\mk{su}(r, r+n), n \geqslant 1$, is isomorphic to $(BC)_r$, and its Dynkin diagram looks like this:

\begin{center}
\begin{dynkinDiagram}[arrow style={length=5pt}, labels*={2,2,2,2,{(2n,1)}}, text style/.style={scale=0.95}, scale=3]B{}
\dynkinRootMark{O}5
\end{dynkinDiagram}
\end{center}

Here we write the simple root multiplicities near the corresponding vertices. Note that keeping track of the root multiplicities is crucial here: the Lie algebra $\mk{sp}(r,r+n), n \geqslant 1$, also has $(BC)_r$ as its restricted root system, but the multiplicities are different.
\end{remark}

Let us look at the homomorphism $\Upomega \colon N_K(\mk{a}) \to \Aut^\mr{w}(\Upsigma)$ through the lens of the semidirect product decomposition $\Aut^\mr{w}(\Upsigma) = \W(\Upsigma) \rtimes \Aut^\mr{w}(\DD)$. First off, note that $\Ker(\Upomega) = Z_K(\mk{a})$. As we have already seen, $\W(\Upsigma) \subseteq \Im(\Upomega)$. In fact, for each reflection $s_\upalpha \in \W(\Upsigma)$, we constructed an element in $K^0$ that preserves $\mk{a}$ and whose image under $\Upomega$ is $s_\upalpha$, which means that $\W(\Upsigma) \subseteq \Upomega(N_{K^0}(\mk{a}))$. Here $N_{K^0}(\mk{a})$ is a subgroup of $N_K(\mk{a})$, and it can also be described as $N_K(\mk{a}) \cap K^0 = N_K(\mk{a}) \cap \Inn(\mk{g})$ (this equality follows from the fact that $K$ is a maximal compact subgroup of $\Aut(\mk{g})$ and thus $K^0 = K \cap \Aut^0(\mk{g}) = K \cap \Inn(\mk{g})$). In particular, $N_{K^0}(\mk{a})$ is a normal subgroup of $N_K(\mk{a})$. It is proven in \cite[Prop. 6.52]{knapp} that the image of $N_{K^0}(\mk{a})$ under $\Upomega$ is precisely $\W(\Upsigma)$ and thus $N_{K^0}(\mk{a})/Z_{K^0}(\mk{a}) \cong \W(\Upsigma)$.

Consider another normalizer subgroup of $K$ given as $N_K(\mk{n})$. Let $k \in N_K(\mk{n})$. As an element of $K$, $k$ commutes with $\uptheta$ and thus preserves $\uptheta \mk{n}$. Since $k$ is orthogonal with respect to $B_\uptheta$, it must preserve $\mk{g} \ominus (\mk{n} \oplus \uptheta \mk{n}) = \mk{g}_0 = \mk{k}_0 \oplus \mk{a}$. But $\mk{k}_0 \subseteq \mk{k}$ and $\mk{k}_0 \perp \mk{a}$, so $k$ preserves $\mk{k}_0 = \mk{g}_0 \cap \mk{k}$ and thus $\mk{a}$. We conclude that $N_K(\mk{n}) \subseteq N_K(\mk{a})$. Since $\Aut^\mr{w}(\DD)$ consists precisely of those weighted automorphisms of $\Upsigma$ that preserve the set of positive roots and $k(\mk{g}_\upalpha) = \mk{g}_{\Upomega(k)(\upalpha)}$, we deduce that $N_K(\mk{n}) = \Upomega^{-1}(\Aut^\mr{w}(\DD))$.

\begin{theorem}\label{maintheorem}
Let $\mk{g}$ be a real semisimple Lie algebra with $\uptheta, \mk{a},$ and $\Upsigma^+$ fixed. Then:
\begin{enumerate}
    \item $\Upomega(N_K(\mk{n})) = \Aut^\mr{w}(\DD)$, and hence $\Upomega$ is surjective.
    \item If $\mk{g}$ is simple, then $\Aut^\mr{w}(\DD) = \Aut(\DD)$, and hence $\Aut^\mr{w}(\Upsigma) = \Aut(\Upsigma)$.
\end{enumerate}
\end{theorem}

Informally, (1) means that every weighted automorphism of $\Upsigma$ can be lifted to an automorphism of $\mk{g}$. Also note that (2) might fail in case $\mk{g}$ is not simple. For instance, if $\mk{g} = \mk{su}(r,r+n) \oplus \mk{sp}(r,r+n), \, n \geqslant 1,$ then $\Upsigma = (BC)_r \sqcup (BC)_r$, so $\Aut(\DD) = \Zz/2\Zz$. But the two connected components of $\DD$ are not weighted-isomorphic, which means that $\Aut^\mr{w}(\DD)$ is trivial.

We will prove Theorem \ref{maintheorem} by first reducing it to the simple case and then (mostly) to the theory of complex semisimple Lie algebras.

We need to account for two things: first, $\mk{g}$ might have compact ideals, which make no contribution to the restricted root system, and second, $\mk{g}$ might have isomorphic noncompact simple ideals. To this end, let us write $\mk{g} = \mk{g}_\mr{c} \oplus \mk{g}_1^{l_1} \oplus \cdots \oplus \mk{g}_k^{l_k}$, where $\mk{g}_\mr{c}$ is the sum of all compact ideals of $\mk{g}$, each $\mk{g}_i$ is a noncompact simple ideal, $\mk{g}_i \not\simeq \mk{g}_j$ for $i \ne j$, and $\mk{g}_i^{l_i}$ simply stands for $\bigoplus_{j=1}^{l_i} \mk{g}_i$. Write $l = \sum_{i=1}^k l_i$ and let $S_l^{l_1, \ldots, l_k}$ stand for the subgroup of the symmetric group $S_l$ consisting of those permutations that permute the first $l_1$ elements with each other, the next $l_2$ elements with each other, etc. Clearly, $S_l^{l_1, \ldots, l_k} \simeq S_{l_1} \times \cdots \times S_{l_k}$. There is an obvious embedding $S_l^{l_1, \ldots, l_k} \hookrightarrow \Aut(\mk{g})$ given by the rule $\upsigma \cdot (X_\mr{c}, (X_s)_{s=1}^l) = (X_\mr{c}, (X_{\upsigma(s)})_{s=1}^l)$ (to be precise, this map is an antihomomorphism of groups). We already know that any Cartan involution on $\mk{g}$ respects the decomposition $\mk{g} = \mk{g}_\mr{c} \oplus \mk{g}_1^{l_1} \oplus \cdots \oplus \mk{g}_k^{l_k}$. Without loss of generality, we choose $\uptheta$ so that it is the same on the isomorphic summands, i.e.\label{dungeon} $\uptheta = \left(\Id_{\mk{g}_\mr{c}}, \prod_{j=1}^{l_1} \uptheta_1, \ldots, \prod_{j=1}^{l_k} \uptheta_k \right)$. This way, we have $\mk{k} = \mk{g}_\mr{c} \oplus \mk{k}_1^{l_1} \oplus \cdots \oplus \mk{k}_k^{l_k}$ and $\mk{p} = \mk{p}_1^{l_1} \oplus \cdots \oplus \mk{p}_k^{l_k}$. For such a choice of $\uptheta$, the image of the embedding $S_l^{l_1, \ldots, l_k} \hookrightarrow \Aut(\mk{g})$ actually lies in $N_K(\mk{a})$. We also have a subgroup $\Aut(\mk{g}_\mr{c}) \times \Aut(\mk{g}_1)^{l_1} \times \cdots \times \Aut(\mk{g}_k)^{l_k} \subseteq \Aut(\mk{g})$. Let $K_i = \Aut(\mk{g}_i)^{\uptheta_i}$ for $1 \leqslant i \leqslant k$. We can choose $\mk{a}$ so that $\mk{a} = \mk{a}_1^{l_1} \oplus \cdots \oplus \mk{a}_k^{l_k}, \, \mk{a}_i \subseteq \mk{p}_i$. We then have $\Upsigma = \Upsigma_1^{l_1} \sqcup \ldots \sqcup \Upsigma_k^{l_k}$, where $\Upsigma_i \subseteq \mk{a}_i$ is the restricted root system of $\mk{g}_i$ and $\Upsigma_i^{l_i}$ simply means $\bigsqcup_{j=1}^{l_i} \Upsigma_i$. In a similar vein, we have an obvious embedding $S_l^{l_1, \ldots, l_k} \hookrightarrow \Aut^\mr{w}(\Upsigma)$ and a subgroup $\Aut^\mr{w}(\Upsigma_1)^{l_1} \times \cdots \times \Aut^\mr{w}(\Upsigma_k)^{l_k} \subseteq \Aut^\mr{w}(\Upsigma)$. Any choice of positive roots for $\Upsigma$ is necessarily the union of those for its irreducible components. We may assume that $\Upsigma^+ = (\Upsigma_1^+)^{l_1} \sqcup \ldots \sqcup (\Upsigma_k^+)^{l_k}$, hence $\Uplambda = \Uplambda_1^{l_1} \sqcup \ldots \sqcup \Uplambda_k^{l_k}$, $D = D_1^{l_1} \times \cdots \times D_k^{l_k}$, and $\DD = \DD_1^{l_1} \sqcup \cdots \sqcup \DD_k^{l_k}$.

\begin{proposition}\label{decompositions}
\begin{enumerate}
    \item The group $\Aut(\mk{g})$ decomposes as a semidirect product
    $$
    \Aut(\mk{g}) = \left[ \Aut(\mk{g}_\mr{c}) \times \Aut(\mk{g}_1)^{l_1} \times \cdots \times \Aut(\mk{g}_k)^{l_k} \right] \rtimes S_l^{l_1, \ldots, l_k}.
    $$
    In particular, we have $\Inn(\mk{g}) = \Inn(\mk{g}_\mr{c}) \times \Inn(\mk{g}_1)^{l_1} \times \cdots \times \Inn(\mk{g}_k)^{l_k}$.
    \item The group $K$ decomposes as a semidirect product
    $$
    K = \left[ \Aut(\mk{g}_\mr{c}) \times K_1^{l_1} \times \cdots \times K_k^{l_k} \right] \rtimes S_l^{l_1, \ldots, l_k}.
    $$
    In particular, we have $K^0 = \Inn(\mk{g}_\mr{c}) \times (K_1^0)^{l_1} \times \cdots \times (K_k^0)^{l_k}$.
    \item The group $N_K(\mk{a})$ decomposes as a semidirect product
    $$
    N_K(\mk{a}) = \left[ \Aut(\mk{g}_\mr{c}) \times N_{K_1}(\mk{a}_1)^{l_1} \times \cdots \times N_{K_k}(\mk{a}_k)^{l_k} \right] \rtimes S_l^{l_1, \ldots, l_k}.
    $$
    In particular, we have $N_{K^0}(\mk{a}) = \Inn(\mk{g}_\mr{c}) \times N_{K_1^0}(\mk{a}_1)^{l_1} \times \cdots \times N_{K_k^0}(\mk{a}_k)^{l_k}$.
    \item The group $N_K(\mk{n})$ decomposes as a semidirect product
    $$
    N_K(\mk{n}) = \left[ \Aut(\mk{g}_\mr{c}) \times N_{K_1}(\mk{n}_1)^{l_1} \times \cdots \times N_{K_k}(\mk{n}_k)^{l_k} \right] \rtimes S_l^{l_1, \ldots, l_k}.
    $$
    \item The group $\Aut^\mr{w}(\Upsigma)$ decomposes as a semidirect product
    $$
    \Aut^\mr{w}(\Upsigma) = \left[ \Aut^\mr{w}(\Upsigma_1)^{l_1} \times \cdots \times \Aut^\mr{w}(\Upsigma_k)^{l_k} \right] \rtimes S_l^{l_1, \ldots, l_k}.
    $$
    \item The group $\W(\Upsigma)$ decomposes as a product
    $$
    \W(\Upsigma) = \W(\Upsigma_1)^{l_1} \times \cdots \times \W(\Upsigma_k)^{l_k}.
    $$
    \item The group $\Aut^\mr{w}(\DD)$ decomposes as a semidirect product
    $$
    \Aut^\mr{w}(\DD) = \left[ \Aut^\mr{w}(\DD_1)^{l_1} \times \cdots \times \Aut^\mr{w}(\DD_k)^{l_k} \right] \rtimes S_l^{l_1, \ldots, l_k}.
    $$
    \item With respect to the decompositions (3) and (5), the homomorphism $\Upomega \colon N_K(\mk{a}) \to \Aut^\mr{w}(\Upsigma)$ decomposes as
    $$
    \Upomega = \left( E, \Upomega_1^{l_1}, \ldots, \Upomega_k^{l_k}, \Id_{S_l^{l_1, \ldots, l_k}} \right),
    $$
    where $E$ is the trivial homomorphism $\Aut(\mk{g}_\mr{c}) \to \set{e}$, $\Upomega_i \colon N_{K_i}(\mk{a}_i) \to \Aut^\mr{w}(\Upsigma_i)$, and the last component $\Id_{S_l^{l_1, \ldots, l_k}}$ formally means that the following diagram commutes:
    $$
    \xymatrix{
    & N_K(\mk{a}) \ar[dd]^{\displaystyle \Upomega} \\
    S_l^{l_1, \ldots, l_k} \ar@{^{(}->}[ur] \ar@{^{(}->}[dr] \\
    & \Aut^\mr{w}(\Upsigma)
    }
    $$
\end{enumerate}
\end{proposition}

We omit the proof, as it simply boils down to the fact that any automorphism of $\mk{g}$ must preserve $\mk{g}_\mr{c}$ and permute the remaining noncompact simple ideals, and the same is true for weighted automorphisms of $\Upsigma = \Upsigma_1^{l_1} \sqcup \ldots \sqcup \Upsigma_k^{l_k}$.

Part (8) of Proposition \ref{decompositions} implies that if $\Upomega_i(N_{K_i}(\mk{n}_i)) = \Aut^\mr{w}(\DD_i)$ for each $i$, then $\Upomega(N_K(\mk{n})) = \Aut^\mr{w}(\DD)$. Consequently, in order to prove Theorem \ref{maintheorem}, we may restrict to the case when $\mk{g}$ is simple and noncompact. We will actually show that $\Upomega(N_K(\mk{n})) = \Aut(\DD)$ in this case, thus proving both parts (1) and (2) of the theorem.

We will consider three different scenarios. To begin with, we can immediately cast aside all those simple Lie algebras where $\Aut(\DD)$ (and hence $\Aut^\mr{w}(\DD)$) is trivial. This leaves us with those Lie algebras where $\Upsigma = A_n \, (n \geqslant 2), D_n \, (n \geqslant 4),$ or $E_6$.

Each complex semisimple Lie algebra $\mk{g}$ gives rise to at least two noncompact real ones: the realification and the split real form of $\mk{g}$. These are going to be our first two scenarios. As a matter fact, here we do not require the Lie algebra to be simple, and we will only use that assumption in the third scenario. Let $\mk{g}$ be a complex semisimple Lie algebra, and let $\mk{h}, \Updelta, \Uplambda,$ and $\set{h_i, e_i, f_i}_{i=1}^r$ be as on page \pageref{complexnotation}. It follows from the Isomorphism Theorem that there exists a unique automorphism $\uptheta$ of $\mk{g}$ as of a real Lie algebra that is $\Cx$-antilinear and satisfies 
\begin{equation}\label{Cartan_generators}
\uptheta(h_i) = -h_i, \quad \uptheta(e_i) = -f_i, \quad \uptheta(f_i) = -e_i.
\end{equation}
This automorphism is involutive and is in fact a compact real structure and a Cartan involution\footnote{The latter two notions coincide for any complex semisimple Lie algebra, see \cite[Sec. 5]{onishchik}} (hence the notation). In particular, $\mk{p} = i\mk{k}$. Every Cartan involution on $\mk{g}$ is of this form (for some choice of $\mk{h}, \Uplambda,$ and canonical generators). We can introduce two more involutive automorphisms of $\mk{g}$: the Weyl involution $\upomega$ and the split real structure $\uptau$. The Weyl involution is given on the canonical generators by the same formula \eqref{Cartan_generators} but is $\Cx$-linear, whereas the split real form fixes all the canonical generators but is $\Cx$-antilinear. Once again, the existence and uniqueness of both of these automorphisms follow from the Isomorphism Theorem. Clearly, the three automorphisms commute pairwise and the product of any two of them equals the third one. The fixed point (real) subalgebra $\mk{g}^\uptau$ is the split real form of $\mk{g}$, and every split real form is of this form (for some choice of $\mk{h}, \Uplambda,$ and canonical generators).

\textbf{Scenario 1: the realification}. Here we assume that our real semisimple Lie algebra is the realification of a complex one and use the notation established in the previous paragraph. In particular, the Cartan involution $\uptheta$ is given on the canonical generators by \eqref{Cartan_generators}. Write $\mk{g} = \mk{k} \oplus \mk{p}$ for the corresponding Cartan decomposition. Since $\uptheta$ is $\Cx$-antilinear and $\uptheta(h_i) = -h_i$ for each $i$, we have $\mk{h} \cap \mk{p} = \mk{h}(\Rl), \, \mk{h} \cap \mk{k} = i\mk{h}(\Rl)$. We claim that $\mk{h}(\Rl)$ is a maximal abelian subspace of $\mk{p}$. Indeed, let $\mk{b} \subset \mk{p}$ be an abelian subspace containing $\mk{h}(\Rl)$. If we think of $\mk{g}$ as a real Euclidean vector space (with respect to the inner product $B_\uptheta$), then all operators of the form $\ad(X), \, X \in \mk{p},$ are self-adjoint (this is true for any real semisimple Lie algebra). Since such operators are also $\Cx$-linear, they are also self-adjoint with respect to the Hermitian inner product $\cross{X}{Y} = B_\uptheta(X,Y) + iB_\uptheta(X, iY)$ and hence diagonalizable over $\Cx$. It follows that $\mk{b}_\Cx = i\mk{b} \oplus \mk{b}$ (here $i\mk{b} \subseteq i\mk{p} = \mk{k}$) is an abelian complex subalgebra of $\mk{g}$ consisting of semisimple elements and containing $\mk{h}$, hence we must have $\mk{b}_\Cx = \mk{h}$ and thus $\mk{b} = \mk{h}(\Rl)$. So we can write $\mk{a} = \mk{h}(\Rl)$. In this case, $\mk{a}^* = \mk{h}^*(\Rl)$, and the root system $\Updelta$ of $\mk{g}$ as of a complex semisimple Lie algebra coincides with the restricted root system $\Upsigma$ of $\mk{g}$ as of a real semisimple Lie algebra (for our specific choice of $\uptheta$ and $\mk{a}$). Moreover, the root space decomposition and the restricted root space decomposition coincide as well. Note that $\mk{g}_0 = \mk{h} = i\mk{a} \oplus \mk{a}$ and $\mk{k}_0 = i\mk{a}$. Each restricted root subspace $\mk{g}_\upalpha, \upalpha \in \Upsigma,$ thus has real dimension 2, i.e. all of the root multiplicities equal 2. As we know from Section \ref{section_complex}, every (not necessarily weighted) automorphism $s \in \Aut(\DD)$ can be lifted to a (complex) automorphism $\hat{s}$ of $\mk{g}$ given by the rule $\hat{s}(h_i) = h_{s(i)}, \hat{s}(e_i) = e_{s(i)}, \hat{s}(f_i) = f_{s(i)}$. This automorphism satisfies $\hat{s}(\mk{a}) = \mk{a}$ and $\Uppsi(\hat{s}) = (\restr{\hat{s}}{\mk{a}}^*)^{-1} = s$. It is easily seen on the canonical generators that $\hat{s}$ commutes with $\uptheta$ and preserves $\mk{n}$. Consequently, $\hat{s} \in N_K(\mk{n})$ and $\Upomega(\hat{s}) = \Uppsi(\hat{s}) = s$, which finishes the proof in this scenario.

\textbf{Scenario 2: split real form}. Now suppose that our real semisimple Lie algebra is a split real form of its complexification. We denote the latter by $\mk{g}$ and use the same notation as above. Then our real semisimple Lie algebra can be described as $\mk{g}^\uptau$. One can easily see that $\mk{g}^\uptau$ is the real Lie subalgebra of $\mk{g}$ generated by $e_i, f_i,$ and $h_i, 1 \leqslant i \leqslant r$. The automorphism $\uptheta$ of $\mk{g}$ clearly preserves $\mk{g}^\uptau$. Since $\mk{g}^\uptau$ is a real form of $\mk{g}$, the Killing form $B$ of $\mk{g}$ is the $\Cx$-bilinear extension of the Killing form $B^\uptau$ of $\mk{g}^\uptau$. This implies that the restriction of $\uptheta$ to $\mk{g}^\uptau$ is a Cartan involution on $\mk{g}^\uptau$. Moreover, if we write the corresponding Cartan decomposition as $\mk{g}^\uptau = \mk{k} \oplus \mk{p}$, then $\mk{h}(\Rl)$ lies in $\mk{p}$ and is its maximal abelian subspace. Indeed, for each $X \in \mk{p}$, the operator $\ad_{\mk{g}^\uptau}(X)$ is diagonalizable over $\Rl$, hence its $\Cx$-linear extension $\ad_{\mk{g}}(X)$ is diagonalizable over $\Cx$ as an operator on $\mk{g}$. Therefore, in the same fashion as above, the existence of a larger abelian subspace of $\mk{p}$ would lead to a toral subalgebra of $\mk{g}$ larger than $\mk{h}$, hence a contradiction. Once again, we can take $\mk{a} = \mk{h}(\Rl)$, in which case $\mk{a}^* = \mk{h}^*(\Rl)$ and the restricted root system $\Upsigma$ of $\mk{g}^\uptau$ coincides with $\Updelta$. Just as before, every diagram automorphism $s \in \Aut(\DD)$ can be lifted to the complex automorphism $\hat{s} \in N_{\Aut(\mk{g})}(\mk{h})$ of $\mk{g}$ such that $\Uppsi(\hat{s}) = s$, and it can be seen from the defining formula for $\hat{s}$ that it commutes with both $\uptheta$ and $\uptau$. In particular, it preserves $\mk{g}^\uptau$ and the restriction $\restr{\hat{s}}{\mk{g}^\uptau}$ lies in $N_K(\mk{n})$. We have $\Upomega(\restr{\hat{s}}{\mk{g}^\uptau}) = \Uppsi(\hat{s}) = s$, which finishes the proof in scenario 2.

\textbf{Scenario 3: the rest}. Now we go back to our assumption that $\mk{g}$ is simple and $\Aut(\DD)$ is nontrivial. An examination of the list of all real simple noncompact Lie algebras (\cite[pp. 336-340]{submanifoldsholonomy}) reveals that if $\mk{g}$ is neither split nor complex, it has to be isomorphic to either $\mk{sl}(n,\Hq) \, (n \geqslant 3)$ or $\mk{e}_6^{-26}$. The restricted root systems of these Lie algebras are $A_{n-1}$ and $A_2$, respectively. In both cases, $\Aut(\DD) \cong \mathbb{Z}/2\mathbb{Z}$, and there is only one nontrivial diagram automorphism that we want to lift to $N_K(\mk{n})$. Recall that we have a distinguished automorphism $\uptheta$ of $\mk{g}$ fixed. Plainly, $\uptheta \in N_K(\mk{a})$ and $\Upomega(\uptheta) = -\Id_{\mk{a}^*}$. The weighted root system automorphism $-\Id_{\mk{a}^*}$ can be decomposed as $-\Id_{\mk{a}^*} = w_0 s$, where $w_0 \in \W(\Upsigma)$ and $s \in \Aut^\mr{w}(\DD)$. Here $s(D) = D$ and $-\Id_{\mk{a}^*}(D) = -D$, so $w_0(D) = -D$ (this uniquely determines $w_0$ and also shows that it is the longest element of $\W(\Upsigma)$ with respect to the system of generators $s_{\upalpha_1}, \ldots, s_{\upalpha_r}$). The diagram automorphism $s = -w_0$ may or may not be trivial, depending on $\Upsigma$. Note that this construction does not really rely on $\mk{g}$, nor does it use root multiplicities, so it can be carried out for any root system $(V, \Updelta)$: pick $\Updelta^+$ and decompose $-\Id_V \in \Aut(\Updelta)$ as $-\Id_V = w_0 s$ with respect to the semidirect product decomposition $\Aut(\Updelta) = \W(\Updelta) \rtimes \Aut(\DD)$. It was shown in \cite[\S 4, Prop. 4]{onishchik} that, in case $\Updelta$ is irreducible, $s$ is a nontrivial diagram automorphism precisely when $\Updelta = A_n \, (n \geqslant 2), D_{2n+1} (n \geqslant 2),$ or $E_6$. This covers both of our cases $\mk{g} = \mk{sl}(n,\Hq) \, (n \geqslant 3)$ and $\mk{g} = \mk{e}_6^{-26}$. Now, we know from the discussion after Proposition \ref{meow} that there exists $\upvarphi \in N_K(\mk{a})$ such that $\Upomega(\upvarphi) = w_0$. We have $\upvarphi \uptheta \in N_K(\mk{n})$ and $\Upomega(\upvarphi \uptheta) = w_0^2 s = s$. In other words, the only nontrivial element of $\Aut(\DD)$ lies in the image of $\Upomega$ and so $\Upomega(N_K(\mk{n})) = \Aut(\DD)$, which completes the proof of Theorem \ref{maintheorem}.

\begin{remark}
Part (2) of Theorem \ref{maintheorem} can also be proven simply by examining the classification of simple noncompact Lie algebras and the list of their (weighted) Dynkin diagrams (\cite[pp. 336-340]{submanifoldsholonomy}). Indeed, for any such diagram, if there are two vertices that differ by a (not necessarily weighted) diagram automorphism, then they happen to have the same multiplicity, which implies that every diagram automorphism is weighted.
\end{remark}

\begin{sepcorollary}[maintheorem]
Let $\mk{g}$ be a real semisimple Lie algebra. Then we have $N_K(\mk{a})/Z_K(\mk{a}) \cong \Aut^\mr{w}(\Upsigma)$ and $N_K(\mk{n})/Z_K(\mk{n}) \cong \Aut^\mr{w}(\DD)$.
\end{sepcorollary}

\begin{sepcorollary}[maintheorem]
Let $\mk{g}, \mk{g}'$ be real semisimple Lie algebras with $\uptheta, \uptheta', \mk{a}, \mk{a}', \Upsigma^+,$ and $\Upsigma'^+$ fixed.
\begin{enumerate}
    \item Every weighted isomorphism $f \colon \mk{a} \isoto \mk{a}^*$ between $\Upsigma$ and $\Upsigma'$ is an isometry. In particular, $\Aut^\mr{w}(\Upsigma) \subseteq \O(\mk{a}^*)$.
    \item Now assume that neither $\mk{g}$ nor $\mk{g}'$ have nonzero compact ideals. Then for every weighted isomorphism $f \colon \mk{a}^* \isoto \mk{a}'^*$ between $\Upsigma$ and $\Upsigma'$, there exists a Lie algebra isomorphism $F \colon \mk{g} \isoto \mk{g}'$ such that $F \circ \uptheta = \uptheta' \circ F$ and $F(\mk{a}) = \mk{a}'$ and the induced weighted isomorphism $\mk{a}^* \isoto \mk{a}'^*$ between $\Upsigma$ and $\Upsigma'$ coincides with $f$. In particular, for every weighted diagram isomorphism $s \colon \Uplambda \isoto \Uplambda'$ between $\DD$ and $\DD'$, there exists such $F \colon \mk{g} \isoto \mk{g}'$ that the induced diagram isomorphism $\restr{f}{\Uplambda} \colon \Uplambda \isoto \Uplambda'$ coincides with $s$.
\end{enumerate}
\end{sepcorollary}

\begin{proof}
For part (2), we know from Remark \ref{woof} that there exists some Lie algebra isomorphism $\wt{F} \colon \mk{g} \isoto \mk{g}'$. As explained in the proof of Proposition \ref{meow}, we may assume $\wt{F} \circ \uptheta = \uptheta' \circ \wt{F}$ and $\wt{F}(\mk{a}) = \mk{a}'$. Let $\wt{f} \colon \mk{a}^* \isoto \mk{a}'^*$ be the induced weighted isomorphism between $\Upsigma$ and $\Upsigma'$. Then $f \circ \wt{f}^{-1} \in \Aut^\mr{w}(\Upsigma')$. According to Theorem \ref{maintheorem}, there exists $\upvarphi' \in N_{K'}(\mk{a}')$ such that $\Upomega'(\upvarphi') = f \circ \wt{f}^{-1}$. Then the weighted root system isomorphism between $\Upsigma$ and $\Upsigma'$ induced by $F = \upvarphi' \circ \wt{F}$ is $\Upomega'(\upvarphi') \circ \wt{f} = f$. For part (1), write $\mk{g} = \mk{g}_\mr{c} \oplus \mk{g}_\mr{nc}$, where $\mk{g}_\mr{c}$ is the maximal compact ideal and $\mk{g}_\mr{nc}$ is the complementary noncompact ideal, and do the same for $\mk{g}'$: $\mk{g}' = \mk{g}'_\mr{c} \oplus \mk{g}'_\mr{nc}$. The restricted root systems of $\mk{g}$ and $\mk{g}'$ coincide with those of $\mk{g}_\mr{nc}$ and $\mk{g}'_\mr{nc}$, respectively, so every weighted isomorphism between them is an isometry by part (2) and our observation after Definition \ref{def3}.
\end{proof}

\section{Symmetric spaces of noncompact type}\label{symmetric_spaces}

The results of the previous section have a useful interpretation in terms of the theory of symmetric spaces of noncompact type, morally because those are 'the same' as noncompact real semisimple Lie algebras. See \cite[Ch. IV-VI]{helgason} for a background on symmetric spaces.

Let $M$ be a symmetric space of noncompact type. It is well known that the isometry group $I(M)$ is a Lie group in the compact-open topology and its action on $M$ is smooth. We denote it by $\widetilde{G}$ and its Lie algebra by $\mk{g}$. We have a map $\mk{g} \to \mk{X}(M)$, called the infinitesimal generator of the (transitive) action $\wt{G} \curvearrowright M$, that sends $X$ to the corresponding fundamental vector field $X^*$. This map is an injective anti-homomorphism of Lie algebras and its image is precisely the Lie subalgebra $\mathcal{K}(M) \subseteq \mk{X}(M)$ of Killing vector fields. Denote $G = \widetilde{G}^0$. Since $M$ is connected, the action $G \curvearrowright M$ is transitive as well.

Pick any $o \in M$, write $\widetilde{K}$ for the isotropy subgroup of $\wt{G}$ at $o$, and denote $K = \wt{K} \cap G$. The action $\wt{G} \curvearrowright M$ (as well as $G \curvearrowright M$) is proper (see \cite{diaz-ramos_proper}), which implies that $\wt{K}$ is a compact subgroup of $\wt{G}$. Moreover, $\wt{G}/\wt{K} = M$ is contractible, so $\wt{K}$ is a maximal compact subgroup of $\wt{G}$, $K$ is a maximal compact subgroup of $G$, and $K = \wt{K}^0$ is connected. We have the geodesic symmetry $s_o \in \wt{G}$ at $o$, which gives rise to an involutive automorphism $\Uptheta = C_{s_o} \colon g \mapsto s_o g s_o^{-1} = s_o g s_o$ of $\wt{G}$ and thus to an involutive automorphism $\uptheta = (C_{s_o})_* = \Ad(s_o)$ of $\mk{g}$. Since $M$ is of noncompact type, $\wt{G}$ (and thus $\mk{g}$) is semisimple and $\uptheta$ is a Cartan involution on $\mk{g}$, so we can write $\mk{g} = \mk{k} \oplus \mk{p}$ for the corresponding Cartan decomposition\footnote{In this section, we are no longer using the letter $K$ to denote $\Aut(\mk{g})^\uptheta$.}. One can show that $\wt{K}$ is an open subgroup of $\wt{G}^\Uptheta$, which means that $\Lie(\wt{K}) = \mk{k}$ and $\Ad_{\wt{G}}(\wt{K}) \subseteq \Aut(\mk{g})^\uptheta$. In particular, the adjoint representation of $\wt{K}$ on $\mk{g}$ is orthogonal with respect to $B_\uptheta$ and preserves $\mk{k}$ and $\mk{p}$.

Consider the orbit map $\wt{G} \twoheadrightarrow M, g \mapsto g \ccdot o$. Its differential at $e$ is a surjective linear map $\uppi \colon \mk{g} \twoheadrightarrow T_oM$ with $\Ker(\uppi) = \mk{k}$, so it maps $\mk{p}$ isomorphically onto $T_oM$. We always identify $\mk{p}$ and $T_oM$ by means of this isomorphism. One can easily show that this is an isomorphism of $\wt{K}$-representations: between the adjoint representation of $\wt{K}$ in $\mk{p}$ and its isotropy representation in $T_oM$.

\begin{remark}\label{base_point}
Note that if we were to pick another point $o' = g \ccdot o \in M$, we would have $s_{o'} = g s_o g^{-1}$ and hence $\uptheta' = \Ad(s_{o'}) = \Ad(g)\uptheta\Ad(g)^{-1}$. Consequently, the corresponding Cartan decompositions differ by $\Ad(g)$: $\Ad(g)(\mk{k}) = \mk{k}', \Ad(g)(\mk{p}) = \mk{p}'$.
\end{remark}

If we denote the Riemannian metric on $M$ by $g$, then $g_o$ is a $\wt{K}$-invariant inner product on $T_oM$, which is the same as an $\Ad_{\wt{G}}(\wt{K})$-invariant inner product on $\mk{p}$. From the previous section, we know that another such inner product is $\restr{B_\uptheta}{\mk{p} \times \mk{p}} = \restr{B}{\mk{p} \times \mk{p}}$. In general, these two inner products may not coincide. However, the difference between them allows an explicit description.

Let $M = M_1 \times \cdots \times M_k$ be the de Rham decomposition of $M$ (see \cite[Ch. IV, Sec. 5,6]{kobayashi_nomizu_I} for more on the de Rham decomposition). We say that $M_i$ and $M_j$ are homothetic and write $\boldsymbol{M_i \sim M_j}$ if there exists a diffeomorphism $\upvarphi \colon M_i \isoto M_j$ such that $\upvarphi^* g_j = a g_i$ for some $a > 0$ (i.e., $\upvarphi$ is 'almost' an isometry: it is conformal with a constant conformal factor). This is an equivalence relation on the set of de Rham factors, and it is in general weaker than $M_i \simeq M_j$ (being isometric). The isometry group of $M$ allows a description similar to that of $\Aut(\mk{g})$ (established in Proposition \ref{decompositions}(1)). We will present it in a slightly different form without grouping isometric de Rham factors together. We define:
\begin{align*}
    \boldsymbol{S_k^\simeq} = \set{\upsigma \in S_k \mid M_i \simeq M_{\upsigma(i)} \; \forall \, i \in \set{1, \ldots, k}}, \\
    \boldsymbol{S_k^\sim} = \set{\upsigma \in S_k \mid M_i \sim M_{\upsigma(i)} \; \forall \, i \in \set{1, \ldots, k}}.
\end{align*}
These are subgroups of the symmetric group $S_k$, and we have $S_k^\simeq \subseteq S_k^\sim$. Note that these definitions only make sense after we fix the (order in the) de Rham decomposition of $M$. For any pair of indices $i,j \in \set{1, \ldots, k}$ such that $M_i \simeq M_j$, pick an isometry $\upvarphi_{ij} \colon M_i \isoto M_j$ in such a way that if we have $M_i \simeq M_j \simeq M_l$, then $\upvarphi_{jl} \circ \upvarphi_{ij} = \upvarphi_{il}$. This gives an embedding 
\begin{gather*}
    S_k^\simeq \hookrightarrow I(M), \upsigma \mapsto \upvarphi_\upsigma, \; \text{where} \\ 
    \upvarphi_\upsigma(p_1, \ldots, p_k) = (\upvarphi_{\upsigma(1)1}(p_{\upsigma(1)}), \ldots, \upvarphi_{\upsigma(k)k}(p_{\upsigma(k)}))
\end{gather*} 
(to be precise, this is an injective group anti-homomorphism). Surely, this embedding depends on the choice of $\upvarphi_{ij}$'s. We also have an obvious embedding $I(M_1) \times \cdots \times I(M_k) \subseteq I(M)$. The following result is proven in \cite{mysecondpaper} (in much greater generality and in a slightly different notation more akin to the one we used in Proposition \ref{decompositions}):

\begin{fact}\label{isometry_semidirect}
Let $M$ be a symmetric space of noncompact type and $M = M_1 \times \dots \times M_k$ its de Rham decomposition. Then the group $I(M)$ decomposes as a semidirect product
    $$
    I(M) = \left[ I(M_1) \times \cdots \times I(M_k) \right] \rtimes S_k^\simeq.
    $$
In particular, we have $I^0(M) = I^0(M_1) \times \cdots \times I^0(M_k)$.

More generally, if $M'$ is another symmetric space of noncompact type with the de Rham decomposition $M' = M'_1 \times \cdots \times M'_k$ and $\upvarphi \colon M \isoto M'$ is an isometry, then there exists a permutation $\upsigma \in S_k$ and a collection of isometries $\upvarphi_i \colon M_i \isoto M'_{\upsigma^{-1}(i)}, 1 \leqslant i \leqslant k$, such that
$$
\upvarphi(p_1, \ldots, p_k) = (\upvarphi_{\upsigma(1)}(p_{\upsigma(1)}), \ldots, \upvarphi_{\upsigma(k)}(p_{\upsigma(k)})).
$$
As a consequence, the de Rham decomposition of $M$ is essentially unique up to reordering of the factors.
\end{fact}

Let us write $\wt{G}_i = I(M_i), G_i = \wt{G}_i^0, \mk{g}_i = \Lie(\wt{G_i})$. Denote $o = (o_1, \ldots, o_k)$ and write $\wt{K}_i$ for the stabilizer subgroup of $\wt{G}_i$ at $o_i$ and $K_i = \wt{K}_i \cap G_i = \wt{K}_i^0$. We necessarily have $s_o = (s_{o_1}, \ldots s_{o_k})$, $\Uptheta = (\Uptheta_1, \ldots, \Uptheta_k)$, and $\uptheta = (\uptheta_1, \ldots, \uptheta_k)$, where $\uptheta_i = (\Uptheta_i)_* = \Ad(s_{o_i})$. The latter decomposition of $\uptheta$ implies that we have $\mk{k} = \bigoplus_{i=1}^k \mk{k}_i$ and $\mk{p} = \bigoplus_{i=1}^k \mk{p}_i$, just like we had on page \pageref{dungeon}. It is well know from the theory of symmetric spaces that $M$ is (de Rham) irreducible if and only if its isometry Lie algebra $\mk{g}$ is simple (this is no longer true for symmetric spaces of compact type). In particular, each $\mk{g}_i$ is simple (and noncompact). We stress that $\mk{g}$ has no nonzero compact ideals. The image of the restricted isotropy representation $K_i \hookrightarrow \O(T_{o_i}M_i)$ coincides with the restricted holonomy group $\Hol^0(M_i,o_i)$ (this is true for all semisimple symmetric spaces). In particular, this representation is irreducible and thus any two $K_i$-invariant inner products on $\mk{p}_i \cong T_{o_i}M_i$ are proportional. We already have two such inner products: $\restr{B}{\mk{p}_i \times \mk{p}_i}$ and $(g_i)_{o_i}$ (here $g_i$ is the Riemannian metric on $M_i$). Therefore, we have $(g_i)_{o_i} = \uplambda_i \restr{B}{\mk{p}_i \times \mk{p}_i}$, and we can write $g_o = \uplambda_1 \restr{B}{\mk{p}_1 \times \mk{p}_1} + \cdots + \uplambda_k \restr{B}{\mk{p}_k \times \mk{p}_k}$. We call $\uplambda_1, \ldots, \uplambda_k$ the \textbf{normalizing constants of} $\boldsymbol{M}$.

\begin{lemma}\label{normalizing_constants}
The normalizing constants of $M$ are well defined up to reordering.
\end{lemma}

\begin{proof}
Formally, the lemma means that if we are given another de Rham decomposition $M = M'_1 \times \cdots \times M'_k$ and base point $o' = (o'_1, \ldots, o'_k)$ with the corresponding normalizing constants $\uplambda'_1, \ldots, \uplambda'_k$, then there exists a permutation $\upsigma \in S_k$ such that $M_i \simeq M'_{\upsigma(i)}$ and $\uplambda_i = \uplambda'_{\upsigma(i)}$ for each $i \in \set{1, \ldots, k}$. According to Fact \ref{isometry_semidirect}, there exists $\upsigma \in S_k$ satisfying the former condition: $M_i \simeq M'_{\upsigma(i)}$. Since each de Rham factor is a symmetric space in its own right, it is in particular a Riemannian homogeneous space, so we can pick isometries $\upvarphi_i \colon M_i \isoto M'_{\upsigma(i)}, 1 \leqslant i \leqslant k$, such that $\upvarphi_i(o_i) = o'_{\upsigma(i)}$. We have a Lie group isomorphism $F_i \colon \wt{G}_i \isoto \wt{G}'_{\upsigma(i)}, g \mapsto \upvarphi_i \circ g \circ \upvarphi_i^{-1},$ such that $F_i(s_{o_i}) = s_{o'_{\upsigma(i)}}$, and thus the induced Lie algebra isomorphism $f_i \colon \mk{g}_i \isoto \mk{g}'_{\upsigma(i)}$ satisfies $f_i(\mk{k}_i) = \mk{k}'_{\upsigma(i)}$ and $f_i(\mk{p}_i) = \mk{p}'_{\upsigma(i)}$. Plainly, $f_i \colon \mk{p}_i \isoto \mk{p}'_{\upsigma(i)}$ is an isometry with respect to the inner products $\restr{B_i}{\mk{p}_i \times \mk{p}_i}$ and $\restr{B'_{\upsigma(i)}}{\mk{p}'_{\upsigma(i)} \times \mk{p}'_{\upsigma(i)}}$, while $d(\upvarphi_i)_{o_i} \colon T_{o_i}M_i \isoto T_{o'_{\upsigma(i)}}M'_{\upsigma(i)}$ is an isometry with respect to $g_{o_i}$ and $g'_{o'_{\upsigma(i)}}$. We have a commutative diagram
$$
\xymatrix{
\mk{p}_i \ar[d]_{\rotatebox{90}{$\sim$}} \ar[rr]^{f_i}_{\sim} && \mk{p}'_{\upsigma(i)} \ar[d]_{\rotatebox{90}{$\sim$}} \\
T_{o_i}M_i \ar[rr]^{d(\upvarphi_i)_{o_i}}_{\sim} && T_{o'_{\upsigma(i)}}M'_{\upsigma(i)}
}
$$
which implies that $\uplambda_i = \uplambda'_{\upsigma(i)}$.
\end{proof}

\begin{proposition}\label{factors}
Given $1 \leqslant i,j \leqslant k$, consider the following conditions:
\begin{enumerate}[(i)]
    \item $M_i$ is isometric to $M_j$.
    \item $M_i$ is homothetic to $M_j$.
    \item $\mk{g}_i \simeq \mk{g}_j$.
\end{enumerate}
Then $(i) \Rightarrow (ii) \Leftrightarrow (iii)$. Moreover, $(ii) \Rightarrow (i)$ if and only if $\uplambda_i = \uplambda_j$.
\end{proposition}

\begin{proof}
Clearly, $(i) \Rightarrow (ii)$ and, since rescaling the Riemannian metric does not change the isometry Lie algebra, $(ii) \Rightarrow (iii)$. To show $(iii) \Rightarrow (ii)$, recall that $M_i$ can be recovered from $\mk{g}_i$ up to an isometry by taking a simply connected Lie group $\widehat{G}_i$ with $\Lie(\widehat{G}_i) \cong \mk{g}_i$, a maximal compact subgroup $\widehat{K}_i \subseteq \widehat{G}_i$, and  endowing the quotient $\widehat{G}_i/\widehat{K}_i$ with a suitable $\widehat{G}_i$-invariant Riemannian metric (there is only one up to rescaling). If $\mk{g}_i \simeq \mk{g}_j$, we can always find a Lie group isomorphism $F \colon \widehat{G}_i \isoto \widehat{G}_j$ such that $F(\widehat{K}_i) = \widehat{K}_j$, which induces a homothety $M_i \to M_j$. 

The proof of the last assertion is very similar to that of Lemma \ref{normalizing_constants}. Assume $(ii)$ and start with any homothety $\upvarphi \colon M_i \to M_j$. Composing it with a suitable isometry of $M_j$ if needed, we can replace $\upvarphi$ with a homothety $\upvarphi' \colon M_i \to M_j$ mapping $o_i$ to $o_j$, and $\upvarphi'$ is an isometry if and only if $\upvarphi$ is. We have an isomorphism $F \colon \wt{G}_i \isoto \wt{G}_j, \uppsi \mapsto \upvarphi' \circ g \circ \upvarphi'^{-1}$, mapping $s_{o_i}$ to $s_{o_j}$. If we write $f = F_* \colon \mk{g}_i \isoto \mk{g}_j$, then $f(\mk{p}_i) = \mk{p}_j$, and the following diagram commutes:
$$
\xymatrix{
\mk{p}_i \ar[d]_{\rotatebox{90}{$\sim$}} \ar[r]^{f}_{\sim} & \mk{p}_j \ar[d]_{\rotatebox{90}{$\sim$}} \\
T_{o_i}M_i \ar[r]^{d\upvarphi'_{o_i}}_{\sim} & T_{o_j}M_j
}
$$
The top arrow is an isometry with respect to the inner product $\restr{B}{\mk{p}_i \times \mk{p}_i}$ and $\restr{B}{\mk{p}_j \times \mk{p}_j}$. Consequently, the bottom arrow is an isometry (i.e., $\upvarphi'$ is an isometry) with respect to the inner product $(g_i)_{o_i} = \uplambda_i \restr{B}{\mk{p}_i \times \mk{p}_i}$ on $T_{o_i}M_i$ and $(g_j)_{o_j} = \uplambda_j \restr{B}{\mk{p}_j \times \mk{p}_j}$ on $T_{o_j}M_j$ if and only if $\uplambda_i = \uplambda_j$, which completes the proof.
\end{proof}

As we already know, condition $(iii)$ in Proposition \ref{factors} is also equivalent to $\Upsigma_i$ and $\Upsigma_j$ (or $\DD_i$ and $\DD_j$) being weighted-isomorphic. Thus, the symmetric space $M$ is fully determined up to an isometry by the (weighted) Dynkin diagram $\DD$ of $\mk{g}$ together with the normalizing constants $\uplambda_1, \ldots, \uplambda_k$ (which we could assign as weights to the connected components $\DD_1, \ldots, \DD_k$ of $\DD$). Proposition \ref{factors} has the following immediate

\begin{corollary}
The following conditions are equivalent:
\begin{enumerate}[(i)]
    \item $\uplambda_i = \uplambda_j$ whenever $M_i \sim M_j$.
    \item $S_k^\simeq = S_k^\sim$.
\end{enumerate}
If these conditions are satisfied, we call the Riemannian metric $g$ on $M$ \textbf{almost Killing}. If, moreover, $\uplambda_1 = \ldots = \uplambda_k = 1$ (i.e. if $g_o = \restr{B}{\mk{p} \times \mk{p}}$), we call $g$ \textbf{Killing}.
\end{corollary}

Note that the Riemannian metric of $M$ is automatically almost Killing if $M$ is irreducible. More generally, if no two distinct de Rham factors of $M$ are homothetic, than its Riemannian metric is almost Killing.

Since the isometry group is not affected by constant rescaling of the metric, Fact \ref{isometry_semidirect} tells us that by rescaling the metric on the de Rham factors of $M$ -- that is, by adjusting the normalizing constants -- we might 'gain' or 'lose' some connected components of $I(M)$ (whereas $I^0(M)$ always stays the same). From this perspective, the almost Killing condition on the Riemannian metric ensures precisely that the isometry group is as large as possible, namely that $I(M) \simeq \left[ I(M_1) \times \cdots \times I(M_k) \right] \rtimes S_k^\sim$.

Having defined the normalizing constants, we can now formulate precisely the correspondence between symmetric spaces of noncompact type and noncompact real semisimple Lie algebras that we mentioned at the beginning of the section. Let $(M,g)$ be a symmetric space of noncompact type with the de Rham decomposition $M = M_1 \times \cdots \times M_k$. Note that if we rescale the Riemannian metric on each de Rham factor $M_i$ by some constant conformal factor $a_i > 0$ and denote the resulting metric on $M$ by $\wt{g}$, then $(M, \wt{g})$ is still a symmetric space of noncompact type and it has the same isometry Lie algebra. We say that $(M,\wt{g})$ is obtained from $(M,g)$ by \textbf{rescaling the normalizing constants}. If $M'$ is another symmetric space of noncompact type, we say that $M$ and $M'$ are equivalent if they become isometric after a suitable rescaling of their normalizing constants. Note that this notion of equivalence is weaker than being homothetic. We can now formulate the aforementioned correspondence:

\begin{equation*}
\left\{
\begin{gathered}
\text{equivalence classes} \\
\text{of symmetric spaces} \\ 
\text{of noncompact type} \\
\end{gathered}
\right\} \xrightarrow[\sim]{M \, \mapsto \, \Lie(I(M))} \left\{
\begin{gathered}
\text{isomorphism classes} \\
\text{of real semisimple Lie algebras} \\ 
\text{without nonzero compact ideals} \\
\end{gathered}
\right\}
\end{equation*}

If we start with a real semisimple Lie algebra $\mk{g}$ without nonzero compact ideals, one way to get a symmetric space $M$ of noncompact type with $\Lie(I(M)) \simeq \mk{g}$ is as follows. Take a simply connected Lie group $G$ with $\Lie(G) \simeq \mk{g}$, take a maximal compact subgroup $K \subset G$, write $M = G/K$, and endow $M$ with any $G$-invariant Riemannian metric. One can show that $I(M) \cong G/\mr{Z}(G)$ and thus $\Lie(I(M)) \simeq \mk{g}/\mk{z}(\mk{g}) = \mk{g}$. We already used this construction in the proof of Proposition \ref{factors}.

We can reformulate part (1) of Proposition \ref{decompositions} in a similar fashion to Fact \ref{isometry_semidirect}. For any pair of indices $i,j \in \set{1, \ldots, k}$ such that $\mk{g}_i \simeq \mk{g}_j$, pick an isomorphism $f_{ij} \colon \mk{g}_i \isoto \mk{g}_j$ in such a way that if we have $\mk{g}_i \simeq \mk{g}_j \simeq \mk{g}_l$, then $f_{jl} \circ f_{ij} = f_{il}$. It follows from Proposition \ref{factors} that this gives an embedding 
\begin{gather*}
    S_k^\sim \hookrightarrow \Aut(\mk{g}), \upsigma \mapsto f_\upsigma, \; \text{where} \\ 
    f_\upsigma(X_1, \ldots, X_k) = (f_{\upsigma(1)1}(X_{\upsigma(1)}), \ldots, f_{\upsigma(k)k}(X_{\upsigma(k)}))
\end{gather*} 
(again, this is an injective group anti-homomorphism). Proposition \ref{decompositions}(1) now asserts that $\Aut(\mk{g}) = \left[ \Aut(\mk{g}_1) \times \cdots \times \Aut(\mk{g}_k) \right] \rtimes S_k^\sim$. We have an open subgroup 
$$
\boldsymbol{\Aut(\mk{g})_M} = \left[ \Aut(\mk{g}_1) \times \cdots \times \Aut(\mk{g}_k) \right] \rtimes S_k^\simeq \subseteq \Aut(\mk{g}),
$$
which may or may not be a proper subgroup, depending on the normalizing constants\footnote{It may be preferable to write $\Aut(\mk{g})_{(M,g)}$ to avoid ambiguity in case one has another metric $\wt{g}$ obtained from $g$ by rescaling the normalizing constants.}. Note that this subgroup does not depend on the choice of $f_{ij}$'s (although the embedding $S_k^\sim \hookrightarrow \Aut(\mk{g})$ surely does depend on this choice). By passing from $\Aut(\mk{g})$ to $\Aut(\mk{g})_M$, we are prohibiting those automorphisms of $\mk{g}$ that permute isomorphic simple ideals whose corresponding normalizing constants do not coincide.

We are now in a position to prove the following result, which characterizes $I(M)$ intrinsically in terms of the Lie algebra $\mk{g}$ -- at least when the metric is almost Killing:

\begin{proposition}\label{isometry_automorphism}
Let $M$ be a symmetric space of noncompact type. The adjoint map $\Ad \colon I(M) \to \Aut(\mk{g})$ is an open embedding of Lie groups with image $\Aut(\mk{g})_M$. Moreover, $\Ad$ is an isomorphism if and only if the Riemannian metric of $M$ is almost Killing. In particular, we always have $\Ad \colon I^0(M) \isoto \Inn(\mk{g})$.
\end{proposition}

\begin{proof}
To begin with, observe that $\Ad$ is a local isomorphism. Indeed, its induced morphism of Lie algebras is $\ad \colon \mk{g} \isoto \Der(\mk{g}) = \Lie(\Aut(\mk{g}))$.

Next we prove that $\Ad$ is injective. Assume that $\upvarphi \in \Ker(\Ad)$. We first show that $\upvarphi(o) = o$. We have $(C_\upvarphi)_* = \Id_\mk{g}$, i.e. $\restr{C_\upvarphi}{G} = \Id_G$, which is the same as to say that $\upvarphi$ commutes with every element of $G$. In particular, it commutes with every element of $K$, which implies that $K$ stabilizes $\upvarphi(o)$. Assume that $\upvarphi(o) \ne o$. Then $K$ fixes every point of a geodesic $\upgamma$ emanating from $o$ and passing through $\upvarphi(o)$. Let $v = \dot{\upgamma}(0) \in T_oM$. We see that $v$ is an invariant of the (restricted) isotropy representation of $K$ in $T_oM$. But it is well known that the subspace of invariants of the (restricted or full) isotropy representation of a symmetric space is contained in (the tangent space to) the flat factor, which $M$ does not have by definition. We deduce that $\upvarphi \in K$. But then $d\upvarphi_o = \restr{\Ad(\upvarphi)}{\mk{p}} = \Id_\mk{p}$, so $\upvarphi = e$. This, together with the previous paragraph, implies that $\Ad$ embeds $I(M)$ into $\Aut(\mk{g})$ as an open subgroup.

Finally, we prove that $\Im(\Ad) = \Aut(\mk{g})_M$ (note that the rest will follow, as this subgroup equals the whole $\Aut(\mk{g})$ if and only if $S_k^\simeq = S_k^\sim$, i.e. if and only if the metric is almost Killing). Recall that if $V$ is a (finite-dimensional) vector space, then the representation of $\GL(V)$ in $V$ admits a unique extension to a representation in the full tensor algebra $TV = \bigoplus_{p,q = 0}^\infty T^{(p,q)}V$ by algebra automorphisms such that on $V^*$ it coincides with the dual representation. We will make use of the following description of the isotropy representation (see exercise A.6 on p. 227 in \cite{helgason} for the statement and p. 564 there for a solution):

\begin{fact}\label{helgason_fact}
Let $M$ be a simply connected symmetric space, $o \in M$, and $\wt{K} \subseteq I(M)$ the full isotropy subgroup at $o$. Then $T \in \GL(T_oM)$ lies in the image of the isotropy representation $\wt{K} \hookrightarrow \GL(T_oM)$ if and only if it preserves the inner product\footnote{This first condition cuts out precisely $\O(T_oM)$.} $g_o \in T^{(0,2)}T_oM$ and the curvature tensor $R_o \in T^{(1,3)}T_oM$.
\end{fact}

Fix $o \in M$ as before and consider the Cartan involution $\uptheta = \Ad(s_o)$ and the subgroup $\Aut(\mk{g})^\uptheta \subseteq \Aut(\mk{g})$ (we used to denote it by $K$ in Section \ref{section_real}). We want to show that $\Aut(\mk{g})^\uptheta$ intersects every connected component of $\Aut(\mk{g})$. Take any $\upeta \in \Aut(\mk{g})$. The automorphism $\upeta\uptheta\upeta^{-1}$ is also a Cartan involution. Since all Cartan involutions are conjugate by inner automorphisms, there exists $\updelta \in \Inn(\mk{g}) = \Aut^0(\mk{g})$ such that $\updelta\upeta\uptheta\upeta^{-1}\updelta^{-1} = \uptheta$, i.e. $\updelta\upeta \in \Aut(\mk{g})^\uptheta$. As $\updelta\upeta$ and $\upeta$ lie in the same connected component of $\Aut(\mk{g})$, we are done.

It then suffices to show that $\Im(\Ad)$ contains\footnote{In order for the RHS in this expression to make sense, we need to put an additional requirement on the choice of $f_{ij}$'s, namely $f_{ij} \circ \uptheta_i = \uptheta_j \circ f_{ij}$ for any $i,j$.} 
$$
\boldsymbol{\Aut(\mk{g})_M^\uptheta} = \Aut(\mk{g})^\uptheta \cap \Aut(\mk{g})_M = \left[ \Aut(\mk{g}_1)^{\uptheta_1} \times \cdots \times \Aut(\mk{g}_k)^{\uptheta_k} \right] \rtimes S_k^\simeq.
$$
Take any element $\upeta$ of this subgroup. It preserves the Cartan decomposition $\mk{g} = \mk{k} \oplus \mk{p}$, so we can write $T = \restr{\upeta}{\mk{p}} \in \GL(\mk{p}) \cong \GL(T_oM)$. We claim that $T$ lies in the image of the isotropy representation $\wt{K} \hookrightarrow \O(T_oM)$. Due to Fact \ref{helgason_fact}, it suffices to show that it is orthogonal and preserves the curvature tensor. As any other automorphism of $\mk{g}$, $\upeta$ is orthogonal with respect to $B$, so $T$ is orthogonal with respect to $\restr{B}{\mk{p} \times \mk{p}}$. By construction, for every $i \in \set{1, \ldots, k}$, if we write $T(\mk{p}_i) = \mk{p}_j$, then $\uplambda_i = \uplambda_j$, so $\restr{T}{\mk{p}_i} \colon \mk{p}_i \isoto \mk{p}_j$ is an isometry with respect to the inner products $(g_i)_{o_i}$ and $(g_j)_{o_j}$, which implies that $T$ is orthogonal with respect to $g_o$ as well.

The fact that $T$ preserves the curvature tensor at $o$ can be readily seen from the well-known expression for $R_o$ under the identification $T_oM \cong \mk{p}$: if $X,Y,Z \in \mk{p}$ then $R_o(X,Y)Z = -[[X,Y],Z]$. We deduce that there exists $k \in \wt{K}$ such that $\Ad(k)$ and $\upeta$ coincide on $\mk{p}$. Since they are both Lie algebra automorphisms, they have to coincide on the subspace $[\mk{p},\mk{p}] \subseteq \mk{k}$ as well. The rest follows from the following

\begin{lemma}
Let $\mk{g}$ be a real semisimple Lie algebra without nonzero compact ideals, and let $\mk{g} = \mk{k} \oplus \mk{p}$ be a Cartan decomposition. Then $[\mk{p},\mk{p}] = \mk{k}$.
\end{lemma}

\begin{proof}[Proof of the lemma]
If $\mk{g}$ is simple, this is the content of \cite[Ch. VI, problems 22-24]{knapp} (see p. 735 for solutions). If $\mk{g}$ is the sum of several noncompact simple ideals, namely $\mk{g} = \bigoplus_{i=1}^k \mk{g}_i$, then $\mk{k} = \bigoplus_{i=1}^k \mk{k}_i, \mk{p} = \bigoplus_{i=1}^k \mk{p}_i$, where $\mk{k}_i = \mk{k} \cap \mk{g}_i, \mk{p}_i = \mk{p} \cap \mk{g}_i$, so we have $[\mk{p}_i, \mk{p}_i] = \mk{k}_i$ and the assertion follows.
\end{proof}
\vspace{-1em}
\end{proof}

\vspace{-1em}

Looking at the proof of Proposition \ref{isometry_automorphism}, we have the following

\begin{sepcorollary}[isometry_automorphism]\label{isotropy_automorphisms}
In the notation of Proposition \ref{isometry_automorphism}, the map $\Ad_{\wt{G}} \colon \wt{K} \hookrightarrow \Aut(\mk{g})^\uptheta$ is an open embedding of Lie groups with image $\Aut(\mk{g})_M^\uptheta$. Moreover, its image is the whole $\Aut(\mk{g})^\uptheta$ if and only if the Riemannian metric is almost Killing. In particular, we always have $\Ad_{\wt{G}} \colon K \isoto \Inn(\mk{g})^\uptheta$ and $\wt{K} = \wt{G}^\Uptheta$ (and thus $K = G^\Uptheta$).
\end{sepcorollary}

\begin{proof}
We need only prove the very last assertion. We compute:
\begin{align*}
\Ad_{\wt{G}}(\wt{G}^\Uptheta) &= \Ad_{\wt{G}}(\wt{G})^\uptheta & &(\Ad_{\wt{G}} \circ \Uptheta = C_\uptheta \circ \Ad_{\wt{G}}) \\
&= \Ad(\wt{G}) \cap \Aut(\mk{g})^\uptheta \\
&= \Aut(\mk{g})_M \cap \Aut(\mk{g})^\uptheta & &(\text{by Proposition \ref{isometry_automorphism}}) \\
&= \Aut(\mk{g})_M^\uptheta \\
&= \Ad_{\wt{G}}(\wt{K}) & &(\text{by the first assertion}),
\end{align*}
so $\wt{G}^\Uptheta = \wt{K}$.
\end{proof}

We can use Proposition \ref{isometry_automorphism} to prove the following simple result, which tells how one can recover $M$ from its isometry Lie algebra $\mk{g}$ in an invariant fashion. Given $\mk{g}$ real semisimple, note that the set $\mathcal{C} \subseteq \Aut(\mk{g})$ of all Cartan involutions on $\mk{g}$ is an immersed submanifold of $\Aut(\mk{g})$, since it is an orbit of the adjoint action of $\Aut(\mk{g})$ on itself.

\begin{proposition}\label{Cartan_involutions}
Let $M$ be a symmetric space of noncompact type and $\mk{g} = \Lie(I(M))$. Then the map $\Upxi \colon M \to \mathcal{C}, p \mapsto \Ad(s_p)$, is a diffeomorphism. 
\end{proposition}

Note that we do not assume the Riemannian metric to be almost Killing here.
\vspace{-0.5em}
\begin{proof}
We know that each $\Ad(s_p)$ is indeed a Cartan involution. If we identify $I(M)$ with a subgroup of $\Aut(\mk{g})$ by means of $\Ad$, $\Upxi$ is easily seen to be $I(M)$-equivariant. As both $M$ and $\mathcal{C}$ are smooth homogeneous $I(M)$-spaces, $\Upxi$ is smooth and surjective. Since we already know that $\Ad$ is injective, it suffices to show that the map $M \to I(M), p \mapsto s_p$, is injective. Given $p \in M$, observe that $p$ is the only fixed point of the symmetry $s_p$ (this is specific to symmetric spaces of noncompact type and utterly fails for symmetric spaces of compact type). Therefore, the symmetries $s_p$ and $s_q$ cannot coincide unless $p = q$, which completes the proof.
\end{proof}
\vspace{-0.5em}
Combining Propositions \ref{decompositions}(1) and \ref{isometry_automorphism}, we obtain the following:

\begin{independentcorollary}\label{maximal_compact}
Let $\mk{g}$ be a real semisimple Lie algebra. For any Cartan involution $\uptheta$ on $\mk{g}$, $\Aut(\mk{g})^\uptheta$ is a maximal compact subgroup of $\Aut(\mk{g})$ and $\Inn(\mk{g})^\uptheta$ is a maximal compact subgroup of $\Inn(\mk{g})$.
\end{independentcorollary}
\vspace{-0.5em}
\begin{proof}
Let us write $\mk{g} = \mk{g}_\mr{c} \oplus \mk{g}_\mr{nc}$, where $\mk{g}_\mr{c}$ is the sum of all compact ideals of $\mk{g}$ and $\mk{g}_\mr{nc}$ is the complementary noncompact semisimple ideal. We know from Section \ref{section_real} that $\uptheta$ respects this decomposition and is the identity on $\mk{g}_\mr{c}$, so we can write $\uptheta = \Id_{\mk{g}_\mr{c}} + \uptheta_\mr{nc}$. According to Proposition \ref{decompositions}(1), $\Aut(\mk{g}) = \Aut(\mk{g}_\mr{c}) \times \Aut(\mk{g}_\mr{nc})$, so $\Aut(\mk{g})^\uptheta = \Aut(\mk{g}_\mr{c}) \times \Aut(\mk{g}_\mr{nc})^{\uptheta_\mr{nc}}$. Take a symmetric space $M$ of noncompact type with $\Lie(I(M)) \cong \mk{g}_\mr{nc}$ and $o \in M$ such that $\Ad(s_o) = \uptheta_\mr{nc}$ (the latter is possible by Proposition \ref{Cartan_involutions}). Rescale the normalizing constants of $M$ if necessary (this would change neither the isometry Lie algebra nor $s_o$) so that the metric becomes almost Killing. By Proposition \ref{isometry_automorphism} and Corollary \ref{isotropy_automorphisms}, we have $\Ad \colon I(M) \isoto \Aut(\mk{g}_\mr{nc})$ and $\Ad_{\wt{G}}(\wt{K}) = \Aut(\mk{g}_\mr{nc})^{\uptheta_\mr{nc}}$, which implies that $\Aut(\mk{g}_\mr{nc})^{\uptheta_\mr{nc}}$ is a maximal compact subgroup of $\Aut(\mk{g}_\mr{nc})$. We are left to prove that $\Aut(\mk{g}_\mr{c})$ is compact. This simply follows from the fact that it is a closed subgroup of $O_{B_\mr{c}}(\mk{g}_\mr{c})$, and the latter is compact because the Killing form $B_\mr{c}$ of $\mk{g}_\mr{c}$ is negative-definite. The last assertion follows easily, for $(\Aut(\mk{g})^\uptheta)^0 = \Aut^0(\mk{g})^\uptheta = \Inn(\mk{g})^\uptheta$ has to be a maximal compact subgroup of $\Aut^0(\mk{g}) = \Inn(\mk{g})$.
\end{proof}
\vspace{-0.5em}
Eventually, we can reformulate the results of Section \ref{section_real} in the language of symmetric spaces. Let $M$ be a symmetric space of noncompact type, and let us keep all the notation as above. Just as we did in Section \ref{section_real}, take $\mk{a} \subseteq \mk{p}$ a maximal abelian subspace, $\Upsigma \subseteq \mk{a}^*$ the corresponding restricted root system, and $\DD$ its Dynkin diagram (for some choice $\Upsigma^+$). Let $M = M_1 \times \cdots \times M_k$ be the de Rham decomposition of $M$. In a similar fashion to $\Aut(\mk{g})_M$, we can define:
\begin{align*}
    \boldsymbol{\Aut^\mr{w}(\Upsigma)_M} &\defeq \prod_{i=1}^k \Aut(\Upsigma_i) \rtimes S_k^\simeq \, \subseteq \, \prod_{i=1}^k \Aut(\Upsigma_i) \rtimes S_k^\sim \simeq \Aut^\mr{w}(\Upsigma), \\
    \boldsymbol{\Aut^\mr{w}(\DD)_M} &\defeq \prod_{i=1}^k \Aut(\DD_i) \rtimes S_k^\simeq \, \subseteq \, \prod_{i=1}^k \Aut(\DD_i) \rtimes S_k^\sim \simeq \Aut^\mr{w}(\DD).
\end{align*}
We are implicitly using Theorem \ref{maintheorem}(2) here by writing $\Aut(\Upsigma_i)$ and $\Aut(\DD_i)$ instead of $\Aut^\mr{w}(\Upsigma_i)$ and $\Aut^\mr{w}(\DD_i)$, respectively.

Just as $\Aut(\mk{g})_M$ can be described as $\Im(\Ad_{\wt{G}})$ due to Proposition \ref{isometry_automorphism}, the groups $\Aut^\mr{w}(\Upsigma)_M$ and $\Aut^\mr{w}(\DD)_M$ allow a neat alternative description as well. Indeed, note that we could endow $\mk{a}^*$ with an alternative inner product by considering $\restr{g_o}{\mk{a} \times \mk{a}}$ and carrying it to $\mk{a}^*$ along the induced isomorphism $\mk{a} \isoto \mk{a}^*$. Let us denote the corresponding orthogonal group by $\O_{g_o}(\mk{a}^*)$. It follows by a straightforward computation that:
\begin{align*}
\Aut^\mr{w}(\Upsigma)_M &= \Aut^\mr{w}(\Upsigma) \cap \O_{g_o}(\mk{a}^*), \\
\Aut^\mr{w}(\DD)_M &= \Aut^\mr{w}(\DD) \cap \O_{g_o}(\mk{a}^*).
\end{align*}
It easily follows from Theorem \ref{maintheorem} and Proposition \ref{decompositions} that: 
\begin{align*}
    &\W(\Upsigma) \subseteq \Aut^\mr{w}(\Upsigma)_M, \\
    &\Aut^\mr{w}(\Upsigma)_M = \W(\Upsigma) \rtimes \Aut^\mr{w}(\DD)_M, \\
    &\Upomega(N_{\Aut(\mk{g})_M^\uptheta}(\mk{a})) = \Aut^\mr{w}(\Upsigma)_M, \\
    &\Upomega(N_{\Aut(\mk{g})_M^\uptheta}(\mk{n})) = \Aut^\mr{w}(\DD)_M.
\end{align*}
Consider the adjoint representation of $\wt{K}$ in $\mk{g}$ and the normalizer $N_{\wt{K}}(\mk{a})$ together with its subgroups $N_{K}(\mk{a})$ and $N_{\wt{K}}(\mk{n})$. It easily follows from Corollary \ref{isotropy_automorphisms} that:
\begin{align*}
    &\Ad(N_{\wt{K}}(\mk{a})) = N_{\Aut(\mk{g})_M^\uptheta}(\mk{a}), \\
    &\Ad(N_{K}(\mk{a})) = N_{\Inn(\mk{g})^\uptheta}(\mk{a}), \\
    &\Ad(N_{\wt{K}}(\mk{n})) = N_{\Aut(\mk{g})_M^\uptheta}(\mk{n}).
\end{align*}
We arrive at the following result, which can be regarded as the geometric version of Theorem \ref{maintheorem}(1):

\begin{independentcorollary}\label{main_theorem_isometry}
Let $M$ be a symmetric space of noncompact type. For every $f \in \Aut^\mr{w}(\Upsigma)_M$, there exists an isometry $k \in N_{\wt{K}}(\mk{a})$ such that $\restr{\Ad(k)}{\mk{a}}^* = f$. If $f \in \Aut^\mr{w}(\DD)_M$, then $k$ necessarily lies in $N_{\wt{K}}(\mk{n})$, and if $f \in \W(\Upsigma)$, then $k$ can be chosen in $N_{K}(\mk{a})$.
\end{independentcorollary}

We finish off with the following application. Corollary \ref{main_theorem_isometry} proves useful when one studies submanifolds of $M$ and wants to understand whether two given submanifolds are isometrically congruent, i.e. one can be mapped onto the other by some global isometry of $M$. One particular class of submanifolds that is especially important consists of so-called boundary components of $M$. We will give only the necessary definitions and will not go into detail as it would require an exposition of the theory of parabolic subgroups and subalgebras. 

Let $\Uplambda$ be a choice of simple roots for $\Upsigma$, and let $\Upphi \subseteq \Uplambda$ be any subset. Write $\Upsigma_\Upphi$ for the root subsystem of $\Upsigma$ spanned by $\Upphi$, and let $\mk{g}_\Upphi$ be the Lie subalgebra of $\mk{g}$ generated by all the root subspaces $\mk{g}_\upalpha$ as $\upalpha$ runs through $\Upsigma_\Upphi$. If we let $G_\Upphi$ be the connected Lie subgroup of $G$ corresponding to $\mk{g}_\Upphi$, then the orbit $B_\Upphi = G_\Upphi \ccdot o$ is a totally geodesic submanifold of $M$ called a boundary component. It is itself a symmetric space of noncompact type of smaller rank ($\rank(B_\Upphi) = |\Upphi|)$. In a sense, boundary components of $M$ are its nicest possible totally geodesic submanifolds.

\begin{proposition}
Let $M$ be a symmetric space of noncompact type. Let $\Upphi_1, \Upphi_2 \subseteq \Uplambda$, and assume that there exists $s \in \Aut^\mr{w}(\DD)_M$ such that $s(\Upphi_1) = \Upphi_2$. Then the boundary components $B_{\Upphi_1}$ and $B_{\Upphi_2}$ are isometrically congruent.
\end{proposition}

\begin{proof}
According to Corollary \ref{main_theorem_isometry}, there exists some $k \in N_{\wt{K}}(\mk{n})$ such that $\Upomega(\Ad(k)) = (\restr{\Ad(k)}{\mk{a}}^*)^{-1} = s$. Since $\Ad(k)(\mk{g}_\upalpha) = \mk{g}_{\Upomega(\Ad(k))(\upalpha)} = \mk{g}_{s(\upalpha)}$, we must have $\Ad(k)(\mk{g}_{\Upphi_1}) = \mk{g}_{\Upphi_2}$, which implies $k G_{\Upphi_1} k^{-1} = G_{\Upphi_2}$ and thus $k(B_{\Upphi_1}) = B_{\Upphi_2}$, which completes the proof.
\end{proof}

\printbibliography

\end{document}